\documentclass{article}

    \usepackage[preprint]{neurips_2021}

\usepackage[utf8]{inputenc} % allow utf-8 input
\usepackage[T1]{fontenc}    % use 8-bit T1 fonts
\usepackage{hyperref}       % hyperlinks
\usepackage{url}            % simple URL typesetting
\usepackage{booktabs}       % professional-quality tables
\usepackage{amsfonts}       % blackboard math symbols
\usepackage{nicefrac}       % compact symbols for 1/2, etc.
\usepackage{microtype}      % microtypography
\usepackage{xcolor}         % colors
\usepackage{natbib}
\usepackage{caption}
\usepackage{subcaption}
\usepackage{wrapfig}

\usepackage{notations}
\usepackage{mystyle}
\usepackage{algorithmic}
\usepackage{algorithm}
\usepackage{bbm}

\newtheorem{theorem}{Theorem}
\newtheorem{definition}{Definition}

\newtheorem{lemma}{Lemma}
\newtheorem{proposition}{Proposition}

\title{The Unbalanced Gromov Wasserstein Distance: Conic Formulation and Relaxation}

\author{%
  Thibault Séjourné\\
  Ecole Normale Supérieure, DMA, PSL\\
  \texttt{thibault.sejourne@ens.fr} \\
   \And
   François-Xavier Vialard \\
   Université Gustave Eiffel\\
   \texttt{francois-xavier.vialard@u-pem.fr} \\
   \And
   Gabriel Peyré \\
   Ecole Normale Supérieure, DMA, CNRS, PSL \\
   \texttt{gabriel.peyre@ens.fr} \\
}

\begin{document}

\maketitle

% !TEX root = ../UnbalancedGW.tex

\begin{abstract}
  Comparing metric measure spaces (i.e. a metric space endowed with a probability distribution) is at the heart of many machine learning problems.
  The most popular distance between such metric measure spaces is the Gromov-Wasserstein (GW) distance, which is the solution of a quadratic assignment problem.
  The GW distance is however limited to the comparison of metric measure spaces endowed with a \emph{probability} distribution.
  To alleviate this issue, we introduce two Unbalanced Gromov-Wasserstein formulations: a distance and a more tractable upper-bounding relaxation.
  They both allow the comparison of metric spaces equipped with arbitrary positive measures up to isometries.
  The first formulation is a positive and definite divergence based on a relaxation of the mass conservation constraint using a novel type of quadratically-homogeneous divergence.
  This divergence works hand in hand with the entropic regularization approach which is popular to solve large scale optimal transport problems. We show that the underlying non-convex optimization problem can be efficiently tackled using a highly parallelizable and GPU-friendly iterative scheme.
  The second formulation is a distance between mm-spaces up to isometries based on a conic lifting.
  Lastly, we provide numerical experiments on synthetic examples and domain adaptation data with a Positive-Unlabeled learning task to highlight the salient features of the unbalanced divergence and its potential applications in ML.
\end{abstract}

% !TEX root = ../neurips_2021.tex

\section{Introduction}

Comparing  data distributions on different metric spaces is a basic problem in machine learning. 
This class of problems is for instance at the heart of surfaces~\citep{bronstein2006generalized} or graph matching~\citep{xu2019scalable} (equipping the surface or graph with its associated geodesic distance), regression problems in quantum chemistry~\citep{gilmer2017neural} (viewing the molecules as distributions of points in $\RR^3$) and natural language processing~\citep{grave2019unsupervised,alvarez2018gromov} (where texts in different languages are embedded as points distributions in different vector spaces).

%%%
\paragraph{Metric measure spaces.}

The mathematical way to formalize these problems is to model the data as \emph{metric measure spaces} (mm-spaces). 
A mm-space is denoted as $\Xx=(X, d, \mu)$ where $X$ is a complete separable set endowed with a distance $d$ and a positive Borel measure $\mu \in \Mm_+(X)$. 
For instance, if $X=(x_i)_i$ is a finite set of points, then $\mu=\sum_i m_i \de_{x_i}$ (here $\de_{x_i}$ is the Dirac mass at $x_i$) is simply a set of positive weights $m_i = \mu(\{x_i\}) \geq 0$ associated to each point $x_i$, which accounts for its mass or importance. For instance, setting some $m_i$ to $0$ is equivalent to removing the point $x_i$.
We refer to~\citet{sturm2012space} for a mathematical account on the theory of mm-spaces.
In all the applications highlighted above, it makes sense to perform the comparisons up to isometric transformations of the data.
Two mm-spaces $\Xx=(X, d_X, \mu)$ and $\Yy=(Y, d_Y, \nu)$ are considered to be equal (denoted $\Xx \sim \Yy$) if they are isometric, meaning that there is a bijection $\psi : \spt(\mu) \rightarrow \spt(\nu)$ (where $\spt(\mu)$ is the support of $\mu$) such that $d_X(x,y) = d_Y(\psi(x),\psi(y))$ and $\psi_\sharp \mu=\nu$. Here $\psi_\sharp$ is the push-forward operator, so that $\psi_\sharp \mu=\nu$ is equivalent to imposing $\nu(A)=\mu(\psi^{-1}(A))$ for any set $A \subset Y$. For discrete spaces where $\mu = \sum_i m_i \de_{x_i}$, then one should have $\nu=\psi_\sharp \mu=\sum_i m_i \de_{\psi(x_i)}$.
As highlighted by \citet{memoli2011gromov}, considering mm-spaces up to isometry is a powerful way to formalize and analyze a wide variety of problems such as matching, regression and classification of distributions of points belonging to different spaces.
Most often, the objects of interest come with a natural distance such as an intrinsic or extrinsic distance and the uniform measure is the usual choice to make mm-spaces widely applicable.
The key to unlock all these  problems is the computation of a distance between mm-spaces up to isometry. So far, existing distances (reviewed below) assume that $\mu$ is a probability distribution, i.e. $\mu(X)=1$. This constraint is not natural and sometimes problematic for most of the practical applications to machine learning. The goal of this paper is to alleviate this restriction. 
We define for the first time a class of distances between unbalanced metric measure spaces,  these distances being upper-bounded by divergences which can be approximated by an efficient numerical scheme.

\paragraph{Csisz\'ar divergences}

The simplest case is when $X=Y$ and one simply ignores the underlying metric. 
One can then use  Csisz\'ar divergences (or $\phi$-divergences), which perform a pointwise comparison (in contrast with optimal transport distances, which perform a displacement comparison). It is defined using an entropy function $\phi: \RR_+ \rightarrow [0, +\infty]$, which is a convex, lo\-wer se\-mi-con\-tinu\-ous, positive function with $\phi(1)=0$. 
The Csisz\'ar $\phi$-divergence reads
$\D_\phi(\mu|\nu) \eqdef \int_X \phi\big(\frac{\d\mu}{\d\nu}\big)\d\nu + \phi^\prime_\infty\int_X \d\mu^\bot$, where $\mu = \frac{\d\mu}{\d\nu} \nu + \mu^\bot$ is called the Radon-Nikodym or the Lebesgue decomposition of $\mu$ with respect to $\nu$
and $\phi^\prime_\infty = \lim_{r \rightarrow \infty}\phi(r) / r \in \RR \cup \{+\infty\}$ is called the recession constant.
%
%\begin{align}\label{eq-defn-csiszar}
%  \D_\phi(\mu|\nu) \eqdef \int_X \phi\pa{\frac{\d\mu}{\d\nu}}\d\nu + \phi^\prime_\infty\int_X \d\mu^\bot, \;\text{where}\;\mu = \frac{\d\mu}{\d\nu} \nu + \mu^\bot \;\text{and}\; \phi^\prime_\infty \eqdef \lim_{r \rightarrow \infty}\frac{\phi(r)}{r}.
%\end{align}
This divergence $\D_\phi$ is convex, positive, 1-homogeneous and weak* lower-semi\-con\-tinu\-ous, see~\citet{liero2015optimal} for details.
%
%One can reverse it in the sense $\D_\phi(\mu|\nu)=\D_\psi(\nu|\mu)$ where $\psi$ is the reverse entropy defined as $\psi(r)=r\phi(1/r)$, $\psi(0)=\phi^\prime_\infty$ and $\psi^\prime_\infty=\phi(0)$.
%
Particular instances of $\phi$-divergences are Kullback-Leibler ($\KL$) for $\phi(r)=r\log(r)-r+1$ (note that $\phi^\prime_\infty=\infty$) and Total Variation ($\TV$) for $\phi(r)=|r-1|$. % and the Hellinger distance for $\phi(r)=(\sqrt{r}-1)^2$.
% 
%
%The indicator divergence $\D_\phi(\mu,\nu)=+\infty$ for $\mu \neq \nu$ and $0$ otherwise is obtained by using $\phi(r)=\iota_{=}(r)$ which is 0 if $r=1$ and $+\infty$ otherwise.

\paragraph{Balanced and unbalanced optimal transport.}

If the common embedding space $X$ is equipped with a distance $d(x,y)$, one can use more elaborated methods such as optimal transport (OT) distances, which are computed by solving convex optimization problems. 
This type of methods has proven useful for ML problems as diverse as domain adaptation~\citep{courty2014domain}, supervised learning over histograms~\citep{frogner2015learning} and unsupervised learning of generative models~\citep{WassersteinGAN}. % OTGAN
In this case, the extension from probability distributions to arbitrary positive measures $(\mu,\nu) \in \Mm_+(X)^2$ is now well understood and corresponds to the theory of unbalanced OT. 
Following~\citet{liero2015optimal,chizat2018unbalanced}, a family of unbalanced Wasserstein distances is defined by solving 
\begin{align}\label{eq-uw}
	\text{UW}(\mu,\nu)^q \eqdef \uinf{\pi \in \Mm(X \times X)} &\int \C(d(x,y)) \d\pi(x,y) + \D_\phi(\pi_1|\mu) + \D_\phi(\pi_2|\mu).
\end{align}
Here $(\pi_1,\pi_2)$ are the two marginals of the joint distribution $\pi$, defined by $\pi_1(A) = \pi(A \times Y)$ for $A \subset X$. 
The mapping $\la : \RR^+ \rightarrow \RR$ and exponent $q\geq 1$ should be chosen wisely to ensure for instance that \text{UW} defines a distance (see Section~\ref{sec-setup-dist}).
It is frequent to take $\rho\D_\phi$ instead of $\D_\phi$ (i.e. take $\psi=\rho\phi$) to adjust the strength of the marginals' penalization.
Balanced OT is retrieved with the convex indicator $\phi=\iota_{\{1\}}$ (i.e. $\phi(1)=0$ and $\phi(x)=+\infty$ otherwise) or by taking the limit $\rho \rightarrow +\infty$, which enforces $\pi_1=\mu$ and $\pi_2=\nu$.
%Classical (balanced) optimal transport is retrieved with $\phi=\iota_{=}$ or by taking the limit $\rho \rightarrow +\infty$, which enforces exact conservation of mass $\pi_1=\mu$ and $\pi_2=\nu$.
%
%In the limit $\rho \rightarrow 0$, when $\D_\phi=\rho\KL$ is the Kullback-Leibler relative entropy divergence, $\text{UW}(\mu,\nu)^2/\rho$ tends to the squared Hellinger distance, which does not introduce any transportation at all.
%
When $0 < \rho < +\infty$, unbalanced OT operates a trade-off  between transportation and creation of mass, which is crucial to be robust to outliers in the data and to cope with mass variations in the modes of the distributions. 
For supervised tasks, the value of $\rho$ should be cross-validated to obtain the best performances. 
Its use is gaining popularity in applications, such as 
%supervised learning~\cite{frogner2015learning}, 
medical imaging registration~\citep{feydy2019fast}, videos~\citep{lee2019parallel}, generative learning~\citep{balaji2020robust} and gradient flow to train neural networks~\citep{chizat2018global,rotskoff2019global}.
Furthermore, existing efficient algorithms for balanced OT extend to this unbalanced problem. In particular Sinkhorn's iterations, introduced in ML for balanced OT by~\citet{CuturiSinkhorn}, extend to unbalanced OT~\citep{chizat2016scaling, sejourne2019sinkhorn}, as detailed in Section~\ref{sec-algo}.

\paragraph{The Gromov-Wasserstein distance and its applications.}

The Gromov-Wasserstein ($\GW$) distance~\citep{memoli2011gromov,sturm2012space} generalizes the notion of OT to the setting of mm-spaces up to isometries. It replaces the linear cost $\int \C(d) \d \pi$ of OT by a quadratic function. It reads 
%\begin{align}
%	&\text{GW}(\Xx,\Yy)^q \eqdef \!\!\!\!
%	\umin{\pi \in \Mm_+(X \times Y)} \enscond{
%		\Gg(\pi)
%	}{ 
%		\begin{matrix}\pi_1=\mu \\ \pi_2=\nu\end{matrix}		
%	},\label{eq-defn-gw}\\
%&\Gg(\pi) \eqdef \int \C( |d_X(x, x') -  d_Y(y, y')| ) \d\pi(x,y) \d\pi(x',y').\nonumber
%\end{align}
\begin{align}\label{eq-defn-gw}
	&\text{GW}(\Xx,\Yy)^q \eqdef \!\!\!\!
	\umin{\pi \in \Mm_+(X \times Y)} \enscond{
		\int \C( |d_X(x, x') -  d_Y(y, y')| ) \d\pi(x,y) \d\pi(x',y')
	}{ 
		\begin{matrix}\pi_1=\mu \\ \pi_2=\nu\end{matrix}		
	}. 
\end{align}
%where the distortion kernel is $\Gamma(x, x',y, y') \eqdef |d_X(x, x') -  d_Y(y, y')|$, with $\Delta$ any distance on $\RR_+$.
%
It is proved in~\citet{memoli2011gromov,sturm2012space} that $\GW$ defines with $\C(t)=t^q$ a distance up to isometries on balanced mm-spaces (i.e. the measures are probability distributions).
%
%In this paper, we extend this construction to arbitrary positive measures, and provide explicit settings in Section~\ref{sec-setup-dist}. 
%
The $\GW$ distance is applied successfully in natural language processing for unsupervised translation learning~\citep{grave2019unsupervised, alvarez2018gromov}, in generative learning for objects lying in spaces of different dimensions~\citep{bunne2019learning} and to build VAE for graphs~\citep{xu2020learning}. It has been adapted for domain adaptation over different spaces~\citep{redko2020co}. It is also a relevant distance to compute barycenters between graphs or shapes ~\citep{vayer2018fused,chowdhury2020gromov}.
When $(\Xx,\Yy)$ are Euclidean spaces, this distance compares distributions up to rigid isometry, and is closely related (but not equal) to metrics defined by procrustes analysis~\citep{grave2019unsupervised, alvarez2019towards}. 
The problem~\eqref{eq-defn-gw} is non convex because the quadratic form $\int \C(|d_X - d_Y|) \d\pi\otimes \pi$ is not positive in general. 
It is in fact closely related to quadratic assignment problems~\citep{burkard1998quadratic}, which are used for graph matching problems, and are known to be NP-hard in general. 
Nevertheless, non-convex optimization methods have been shown to be successful in practice to use $\GW$ distances for ML problems. This includes for instance alternating minimization~\citep{memoli2011gromov,redko2020co} and entropic regularization~\citep{peyre2016gromov, gold1996graduated}.

\paragraph{Related works and contributions.}
The concomitant work of~\citet{de2020entropy} extends the $L^p$ transportation distance defined in~\citet{sturm2006geometry} to unbalanced mm-spaces and studies its geometric properties.
This distortion distance is not equivalent to the $\GW$ distance, and is more difficult to estimate numerically because it explicitly imposes a triangle inequality constraint in the optimization problem. 
The work of~\citet{chapel2020partial} relaxes the $\GW$ distance to the unbalanced setting by hybridizing $\GW$ with partial OT~\citep{figalli2010optimal} for unsupervised labeling. It ressembles one particular setting of our formulation, but with some important differences, detailed in Section~\ref{sec-distance}.
Our construction is also connected to partial matching methods, which find numerous applications in graphics and vision~\citep{cosmo2016shrec}. In particular, \citet{rodola2012game} introduces a mass conservation relaxation of the $\GW$ problem. 

The two main contributions of this paper are the definition of two formulations relaxing the $\GW$ distance. The first one is called the Unbalanced Gromov-Wasserstein ($\UGW$) divergence and can be computed efficiently on GPUs. 
The second one is called the Conic Gromov-Wasserstein distance ($\CGW$). 
It is proved to be a distance between mm-spaces endowed with positive measures up to isometries, as stated in Theorem~\ref{thm-ugw-dist} which is the main theoretical result of this paper. We also prove in Theorem~\ref{thm-ugw-dist} that $\UGW$ can be used as a surrogate upper-bounding $\CGW$.  
We present those concepts and their properties in Section~\ref{sec-distance}.
We also detail in Section~\ref{sec-algo} an efficient computational scheme for a particular setting of $\UGW$. This method computes an approximate stationary point of a biconvex relaxation of our formulations. Even though it is a lower bound of the original problem, we provide in Theorem~\ref{ThTightLowerBound} conditions ensuring the tightness of this relaxation in many cases of interest. The algorithm leverages the strength of entropic regularization and the Sinkhorn algorithm, namely that it is GPU-friendly and defines smooth loss functions amenable to back-propagation for ML applications.
Section~\ref{sec-xp} provides some numerical experiments to highlight the qualitative behavior of this algorithm and its ability to cope with outliers and mass variations in the modes of the distributions.
We illustrate numerically the tightness of the relation between $\UGW$ and $\CGW$, showing that $\UGW$ is a reasonnable proxy of a distance, at least locally.
We provide an application of our divergence in the positive unlabeled learning setting, using domain adaptation data, and display results which are at par or outperform the computable competitor~\citet{chapel2020partial}.

% !TEX root = ../neurips_2021.tex

\section{Unbalanced Gromov-Wasserstein formulations}
\label{sec-distance}

We present in this section our two new formulations and their properties. The first one, called UGW, is exploited in Sections~\ref{sec-algo} and~\ref{sec-xp} to derive an efficient algorithm used in numerical experiments. The second one, called CGW, defines a distance between mm-spaces up to isometries. Those results build upon the work of~\citet{liero2015optimal}, and a summary of the construction of UOT is detailed in Appendix~\ref{sec-app-background}.
In all what follows, we consider complete separable mm-spaces endowed with a metric and a positive measure.

%%%%%%%%%%%%%%%%%%%%%%%%%%%%%%%%%%%%%%%%%%%%%%%%%%%%%%%%%%%%%%%
\subsection{The unbalanced Gromov-Wasserstein divergence}

This new formulation makes use of quadratic $\phi$-divergences, defined as $\D_\phi^\otimes(\rho|\nu) \eqdef \D_\phi(\rho \otimes \rho|\nu \otimes \nu)$, where $\rho \otimes \rho \in \Mm_+(X^2)$ is the tensor product measure defined by $\d(\rho \otimes \rho)(x,y)=\d\rho(x)\d\rho(y)$.
Note that $\D_\phi^\otimes$ is not a convex function in general.

\begin{definition}[Unbalanced GW]
The Unbalanced Gromov-Wasserstein divergence is defined as $\UGW(\Xx, \Yy) = \inf_{\pi\in\Mmp(X\times Y)} \Ll(\pi)\eqdef \Gg(\pi)
   + \D_\phi^\otimes(\pi_1|\mu) + \D_\phi^\otimes(\pi_2|\nu)$.
  %\begin{equation}\label{eq-ugw-func}
%\begin{aligned}  \Ll(\pi) \eqdef \Gg(\pi) + \D_\phi^\otimes(\pi_1|\mu) + \D_\phi^\otimes(\pi_2|\nu)$.\end{aligned}
%\end{equation}
\end{definition}

This definition can be understood as an hybridation between~\eqref{eq-uw} and~\eqref{eq-defn-gw} but with a twist: one needs to use the quadratic divergence $\D_\phi^\otimes$ in place of $\D_\phi$. To the best of our knowledge, it is the first time such quadratic divergences are being used and studied.
In the TV case, this is the most important distinction between $\UGW$ and partial-$\GW$~\citep{chapel2020partial}.
Note also that the balanced $\GW$ distance~\eqref{eq-defn-gw}  is recovered as a particular case when using $\phi=\iota_{\{1\}}$ or by letting $\rho \rightarrow +\infty$ for an entropy $\psi=\rho\phi$.

Using quadratic divergences results in UGW being 2-homogeneous: 
%\redrm{if $\mu$ and $\nu$ are multiplied by $\th \geq 0$, then $\UGW(\Xx, \Yy)$ is multiplied by $\th^2$.} 
%
for $\th\geq 0$, writing $(\Xx_\th, \Yy_\th)$ equiped with $(\th\mu,\th\nu)$, one has $\th^{-2}\UGW(\Xx_\th, \Yy_\th)=\UGW(\Xx, \Yy)$.
%
%\blue{Using $\D_\phi$ instead of $\D_\phi^\otimes$ makes the problem degenerate: writing $\tilde{\UGW}$ the problem penalized with $\D_\phi$, one has $\lim_{\th\rightarrow\infty} \th^{-2}\tilde{\UGW}(\Xx_\th, \Yy_\th)=0$ and $\lim_{\th\rightarrow\infty} \th^{-1}\tilde{\UGW}(\Xx_\th, \Yy_\th)=+\infty$}.
%
When using non tensorized $\phi$-divergences, the resulting unbalanced Gromov-Wassertein functional between $\Xx_\th$ and $\Yy_\th$ have very different and inconsistent behaviors when $\theta\rightarrow 0$ and $\theta\rightarrow +\infty$. Indeed, once normalized by $\th^{-2}$ and $\th^{-1}$, one obtains respectively balanced $\GW$ and a Hellinger-type distance. Using tensorized divergences ensures that the behavior does not depends on $\th$. It is also fundamental to connect $\UGW$ with our distance $\CGW$, see Theorem~\ref{thm-ugw-dist} and Appendix~\ref{appendix-distance-ugw}.
%
%\redrm{Not using such divergences results in a lack of} \blue{The} 2-homogeneity of $\Ll$ is crucial to show the connection between our proposed formulations.

We first prove the existence of optimal plans $\pi$ minimizing $\Ll$, which holds for the three key settings of Section~\ref{sec-setup-dist}, namely for $\KL$, $\TV$, and for compact metric spaces (such as finite pointclouds and graphs). All proofs are deferred in Appendix~\ref{appendix-distance-ugw}.

\begin{proposition}[Existence of minimizers]\label{thm-exist-minimizer}
	We assume that $(X,Y)$ are compact and that either (i) $\phi$ superlinear, i.e $\phi^\prime_\infty=\infty$, or (ii) $\C$ has compact sublevel sets in $\RR_+$ and $2\phi^\prime_\infty + \inf \C >0$.
	Then there exists $\pi\in\Mm_+(X\times Y)$ such that $\UGW(\Xx,\Yy)=\Ll(\pi)$.
\end{proposition}
The following proposition ensures that the functional $\UGW$ can be used to compare mm-spaces.
\begin{proposition}[Definiteness of $\UGW$]\label{prop-ugw-definite}
Assume that $\phi^{-1}(\{0\})=\{1\}$ and $\C^{-1}(\{0\})=\{0\}$. Then $\UGW(\Xx,\Yy) \geq 0$ and is $0$ if and only if $\Xx\sim\Yy$.
\end{proposition}
%\begin{proof}
%	Assume $\UGW(\Xx,\Yy)=0$. Positivity implies that all the terms appearing in $\Ll$ are zero. Similarly to the balanced case~\cite{memoli2011gromov}, the distortion being zero imposes that the plan $\pi$ defines an isometry. The $\phi$-divergences being zero implies that $\pi$ has marginals equal to $(\mu,\nu)$. This plan thus defines an isometric bijection between $\Xx$ and $\Yy$, see Appendix~\ref{appendix-distance-ugw} for details.
%\end{proof}

We end this section with a reformulation of $\UGW$ which is important to make the connection with the second formulation $\CGW$ of the following section.
%Its proof is deferred to Appendix~\ref{appendix-distance-ugw}. 
%
It splits $\UGW$ into two parts: the term $\phi(0)(|(\mu\otimes\mu)^\bot| + |(\nu\otimes\nu)^\bot|)$  accounts for the pure creation/destruction of mass and a new transport cost $L_c$ accounts for the remaining part (partial/pure transport and partial creation/destruction of mass).

\begin{lemma}\label{lem-rewrite-ugw}
	Defining $L_{c}(a,b) \eqdef c + a\phi(1/a) + b\phi(1/b)$, and writing $(\f \eqdef \frac{\d\mu}{\d\pi_1}, \g \eqdef \frac{\d\nu}{\d\pi_2})$ the Lebesgue densities of $(\mu,\nu)$ w.r.t. $(\pi_1,\pi_2)$ such that
	$\mu = \f \pi_1 + \mu^\bot$ and $\nu = \g \pi_2 + \nu^\bot$, one has
	\begin{equation}
		\begin{aligned}\label{eq-rewrite-ugw}
		\Ll(\pi)= &\int_{X^2 \times Y^2} L_{\C(|d_X - d_Y|)}(\f\otimes\f,\g\otimes\g)\d\pi\d\pi +\phi(0)(|(\mu\otimes\mu)^\bot| + |(\nu\otimes\nu)^\bot|).
		\end{aligned}
	\end{equation}
\end{lemma}
\begin{proof}
	Write $\f = \tfrac{\d\mu}{\d\pi_1}$ and $\g = \tfrac{\d\nu}{\d\pi_2}$. 
	The Lebesgue decompositions read
	$\mu\otimes\mu = (\f\otimes\f) \pi_1\otimes\pi_1 + (\mu\otimes\mu)^\bot$ and $\nu\otimes\nu = (\g\otimes\g) \pi_2\otimes\pi_2 + (\nu\otimes\nu)^\bot$, thanks to the tensorized structure of the decomposed plans.
	To prove Equation~\eqref{eq-rewrite-ugw}, we need to define the reverse entropy~\cite{liero2015optimal} such that $\D_\phi(\al|\mu) = \D_\psi(\mu|\al)$, where $\psi(x)\triangleq x\phi(\tfrac{1}{x})$ is also an entropy function satisfying $\psi^\prime_\infty=\phi(0)$.
	One then has
	\begin{align*}
	\Ll(\pi) &= \int_{X^2 \times Y^2} \C(\Gamma)\d\pi\d\pi
	+ \D_\phi^\otimes(\pi_1|\mu) + \D_\phi^\otimes(\pi_2|\nu)\\
	&= \int_{X^2 \times Y^2} \C(\Gamma)\d\pi\d\pi
	+ \D_\psi^\otimes(\mu|\pi_1) + \D_\psi^\otimes(\nu|\pi_2)
	\end{align*}
	\begin{align*}
	\Ll(\pi) &= \int_{X^2 \times Y^2} \C(\Gamma)\d\pi\d\pi + \int_{X^2}\psi(\f\otimes\f)\d\pi_1\d\pi_1 + \int_{Y^2}\psi(\g\otimes\g)\d\pi_2\d\pi_2\nonumber\\
	&\qquad+\phi(0)(|(\mu\otimes\mu)^\bot| + |(\nu\otimes\nu)^\bot|)\\
	&= \int_{X^2 \times Y^2} L_{\C(\Gamma)}(\f\otimes\f,\g\otimes\g)\d\pi\d\pi+\phi(0)(|(\mu\otimes\mu)^\bot| + |(\nu\otimes\nu)^\bot|).\nonumber
	\end{align*}
	Using the definition of $\psi$ in $L_c$ ends the proof.
\end{proof}

%%%%%%%%%%%%%%%%%%%%%%%%%%%%%%%%%%%%%%%%%%%%%%%%%%%%%%%%%%%%%%%
\subsection{The conic Gromov-Wasserstein distance}

We introduce a second ``conic'' formulation of unbalanced $\GW$, which is connected to $\UGW$, and whose construction is inspired by the conic formulation of UOT (see Appendix~\ref{sec-app-background} for an overview). % The main contribution is Theorem~\ref{thm-ugw-dist} which states that it defines a distance between mm-spaces equipped with arbitrary positive measures.

\subsubsection{Background on cone sets and distances}
\label{sec-setup-dist}

The conic formulation lifts a point $x \in X$ to a couple $(x,r) \in X \times \RR^+$ where $r$ encodes some (power of a) mass.
Then we seek optimal transport plans defined over $\Co[X] \eqdef X\times\RR_+ / (X\times\{0\})$, where coordinates $(x,r=0)$ with no mass are merged into a single point $\zc_{X}$ called the apex of the cone. In the sequel, points of $X\times\RR_+$ are noted $(x,r)$, while $[x,r]$ are quotiented points of $\Co[X]$. %are noted .

While transport plans depend on variables  $([x,r],[y,s])$ and $([x',r'],[y',s'])$ in $\Co[X]\times\Co[Y]$, the transportation cost involved in our conic formulation only makes use of the 2-D cone $\Co[\RR_+]$ over $\RR_+$ endowed with the distance $|u-v|$ (note that any other distance on $\RR$ could be used as well).
More specifically, we consider coordinates of the form $([u,a],[v,b]) = ([d_X(x,x'),rr'], [d_Y(y,y'),ss']) \in \Co[\RR_+] \times \Co[\RR_+]$.
Thus we now describe conic discrepancies $\Dd$ on $\Co[\RR_+]$, which are defined for $(p,q) \geq 1$ as $\Dd([u,a], [v,b])^q \eqdef H_{\C(|u-v|)}(a^p, b^p)$, 
where $H_c(a^p, b^p) \eqdef \inf_{\theta\geq 0} \theta L_c(\tfrac{a^p}{\theta}, \tfrac{b^p}{\theta})$
%\begin{equation}
%	\begin{aligned}
%	\Dd([u,a], [v,b])^q \eqdef H_{\C(|u-v|)}(a^p, b^p)\,,
%	\qwhereq
%	H_c(a^p, b^p) \eqdef \inf_{\theta\geq 0} \theta L_c(\tfrac{a^p}{\theta}, \tfrac{b^p}{\theta})\,,
%	\end{aligned}
%\end{equation}
%
is the perspective transform of $L_c$ introduced in Lemma~\ref{lem-rewrite-ugw}.
The intuition underpinning the definition of this cost is that the perspective transform accounts for the possibility to rescale a transport plan $\pi$ by a scalar $\th$ 
%\redrm{at each points in the cone} 
but the scaling is performed pointwise instead of globally.
In general $\Dd$ is not a distance, but it is always definite as stated by this result proved in Appendix~\ref{sec-app-background}.

\begin{proposition}\label{prop-cone-dist-definite}
	Assume $\C^{-1}(\{0\})=\{0\}$, $\phi^{-1}(\{0\})=\{1\}$ and $\phi$ is coercive. 
	%Assume also that for any $(a,b)$, there always exists $\theta$ such that $H_c(a,b)= \theta L_c(\tfrac{a}{\theta}, \tfrac{b}{\theta})$.
	Then $\Dd$ is definite on $\Co[\RR^+]$, i.e. $\Dd([u,a], [v,b]) =0$ if and only if $(a=b=0) \;\textrm{or}\; (a=b \;\textrm{and}\; u=v)$.
\end{proposition}
% The function $H_c$ can be computed in closed form for a certain number of common entropies $\phi$, and we refer to~\citet[Section 5]{liero2015optimal} for an overview.
%
Of particular interest are those $\phi$ where $\Dd$ is a distance, which necessitates a careful choice of $\C,p$ and $q$.
We now detail three examples where this is the case.
% In each setting we provide $(\D_\phi, \C, p, q)$ and its associated cone distance $\Dd$. We focus here on the Euclidean distance, but it generalizes to any metric $d(u,v)$.
\vspace*{-0.35cm}
\paragraph{Gaussian Hellinger distance (GH).}
When $\D_\phi = \KL$, $\C(t) = t^2$ and $q=p=2$, then one has
%\begin{align*}
	$\Dd([u,a], [v,b])^2 = a^2 + b^2 - 2abe^{-|u-v| / 2}$. This cone distance~\citep{burago2001course} is further generalized by~\cite{de2019metric} who shows that $\Dd$ is a distance for power entropies $\phi(s)=\frac{s^p - p(s-1)-1}{p(p-1)}$ if $p \geq 1$ (the case $p=1$ corresponding to $\D_\phi=\KL$).
%\end{align*}
\vspace*{-0.3cm}
\paragraph{Hellinger-Kantorovich (HK) / Wasserstein-Fisher-Rao distance (WFR).} When $\D_\phi =\KL$, $\C(t) = -\log\cos^2(t\wedge\tfrac{\pi}{2})$ and $q=p=2$, then one has
%\begin{align*}
	$\Dd([u,a], [v,b])^2 = a^2 + b^2 - 2ab\cos(\tfrac{\pi}{2}\wedge |u-v|)$.
%\end{align*}
This construction, which might seem peculiar, corresponds to the one used to make unbalanced OT a geodesic distance, as detailed in~\citep{liero2015optimal,chizat2018unbalanced}. 
\vspace*{-0.3cm}
\paragraph{Partial optimal transport distance (PT).} When $\D_\phi = \TV$, $\C(t)=t^q$, $q\geq 1$ and $p=1$, then
%\begin{align*}
	$\Dd([u,a], [v,b])^q = a + b - (a\wedge b)(2-|u-v|^q)_+$ 
%\end{align*}
defines a cone distance~\citep{chizat2018unbalanced}.
% The setting of HK/WFR might seem peculiar, but it makes unbalanced OT a length space induced by the GH distance when $d$ is a geodesic distance on $X$. PT and HK/WFR introduce cut-offs, i.e. there is no transport between points too far from each other, e.g. in the latter case $\C(d(x,y))=+\infty$ if $d(x,y)>\pi/2$. The cut-off can be modified by scaling $\la \mapsto \la(\cdot/s)$ for some cutoff $s$.

\subsubsection{Definitions and properties}

The conic formulation consists in solving a $\GW$ problem on the cone, with the addition of two linear constraints. Informally speaking, $L_c$ from Lemma~\ref{lem-rewrite-ugw} becomes $\Dd$, the term $(|(\mu\otimes\mu)^\bot| + |(\nu\otimes\nu)^\bot|)$ is taken into account by the constraints~\eqref{eq-const-conic} below, and the variables $(\f,\g)$ are replaced by $(r^p,s^p)$.
It reads  $\CGW(\Xx,\Yy) \eqdef \inf_{\al\in\Uu_p(\mu,\nu)}\Hh(\al)$ where
\begin{equation}\label{eq-ugw-conic}
\begin{aligned}
	 \Hh(\al)\eqdef\int &\Dd([d_X(x,x'), r r'], [d_Y(y,y'), s s'])^q \,\d\al([x,r], [y,s])\d\al([x',r'], [y',s']),
\end{aligned}
\end{equation}
and $\Uu_{p}(\mu,\nu)$ is defined as the set
\begin{equation}\label{eq-const-conic}
\Uu_{p}(\mu,\nu) \eqdef \left\{
\begin{aligned}
	 \al\in\Mm_+(\Co[X]\times\Co[Y]),\;
	 \int_{\RR_+} r^p \d\al_1(\cdot,r)=\mu,\;
	 \int_{\RR_+} s^p \d\al_2(\cdot,s)=\nu
\end{aligned}
\right\}.
\end{equation}
It is similar to the conic formulation of UW, see Appendix~\ref{sec-app-background}.
Note that similarly to the $\GW$ formulation~\eqref{eq-defn-gw} -- and in sharp contrast with the conic formulation of UW -- here the transport plans are defined on the cone $\Co[X] \times \Co[Y]$ but the cost $\Dd$ is a distance on $\Co[\RR_+]$.

We present now the main contributions of this paper, proved in Appendix~\ref{appendix-distance-cgw}. 
We state that $\CGW$ defines a distance under conditions that hold for the settings of Section~\ref{sec-setup-dist}, and that it is upper-bounded by $\UGW$. 
%
%While the distance $\CGW^{1/q}$ cannot be casted as a finite dimensional program even in discrete settings (because it is defined on a lifted space),
%
The divergence $\UGW$ can be approximated with efficient numerical schemes as detailed in Section~\ref{sec-algo}.

\begin{theorem}\label{thm-ugw-dist}
		(i) The divergence $\CGW$ is symmetric, positive and definite up to isometries.
		(ii) If $\Dd$ is a distance on $\Co[\RR_+]$, then $\CGW^{1/q}$ is a distance on the set of mm-spaces up to isometries.
		(iii) For any $(\D_\phi, \C, p, q)$ with associated cost $\Dd$ on the cone, one has $\UGW\geq\CGW$.
\end{theorem}

\section{Algorithms}
\label{sec-algo}

We focus in this section on the numerical computation of the upper bound $\UGW$ using a bi-convex relaxation and derive an alternate minimization scheme coupled with entropic regularization.
We also propose to approximate $\CGW$ by doing a similar alternate minimization, as detailed in Appendix~\ref{sec-app-xp}.
We provide guarantees of tightness on the bi-convex relaxation for $\CGW$ (see Theorem~\ref{ThTightLowerBound}).
The computation of the distance $\CGW$ is heavy in practice because it requires an optimization over a lifted conic space, which needs to be discretized. 
Thus it does not scale to large problem for $\CGW$, but allows to explore numerically how tight is the upper bound $\UGW\geq\CGW$, see Section~\ref{sec-xp}.
The algorithm for $\UGW$ is presented on arbitrary measures, the special case of discrete measures being a particular case. 
The discretized formulas and algorithms are detailed in Appendix~\ref{appendix-algo}, see also~\citet{chizat2016scaling,peyre2016gromov}.
All implementations are available at \url{https://github.com/thibsej/unbalanced_gromov_wasserstein}, and installable in Python with the command \textit{pip install unbalancedgw}.

\subsection{Bi-convex relaxation and tightness}

In order to derive a simple numerical approximation scheme, following~\citet{memoli2011gromov}, we introduce a lower bound obtained by introducing two transportation plans. To further accelerate the method and enable GPU-friendly iterations, similarly to~\citet{gold1996softmax,solomon2016entropic}, we consider an entropic regularization. It reads, for any $\epsilon \geq 0$,
\begin{align}\label{eq-lower-bound}
	&\UGW_\epsilon(\Xx, \Yy) \eqdef \inf_\pi\Ll(\pi) +\epsilon\KL^\otimes(\pi|\mu\otimes\nu)
	  \geq\uinf{\pi,\ga} \Ff(\pi,\ga) +\epsilon\KL(\pi\otimes\gamma|(\mu\otimes\nu)^{\otimes 2}),\\
		&\qandq \Ff(\pi,\ga) \eqdef \int_{X^2 \times Y^2} \C(|d_X - d_Y|)\d\pi \otimes \ga
	+ \D_\phi(\pi_1\otimes\ga_1|\mu\otimes\mu) + \D_\phi(\pi_2 \otimes \ga_2|\nu \otimes \nu)\,,\nonumber
\end{align}
where $(\ga_1,\ga_2)$ denote the marginals of the plan $\ga$. In the sequel we write $\Ff_\epsilon = \Ff + \epsilon\KL^\otimes$.
Note that in contrast to the entropic regularization of $\GW$~\cite{peyre2016gromov}, here we use a tensorized entropy to maintain the overall homogeneity of the energy.
A simple method to approximate this lower bound is to perform an alternate minimization on $\pi$ and $\ga$, which is known to converge for smooth $\phi$ to a stationary point since the coupling term in the functional is smooth~\citep{tseng2001convergence}.
%
%We focus now on the structure of our program~\eqref{eq-lower-bound}. We start with a key decomposition property of $\KL^\otimes$.
%
%\begin{proposition}\label{prop-decompose-kl}
%	For any $(\mu,\nu, \al,\be)\in\Mm_+(\Xx)$, one has
%	\begin{equation*}
%	\begin{aligned}
%	\KL(\mu\otimes\nu|\al\otimes\be) &= m(\nu)\KL(\mu|\al) +  m(\mu)\KL(\nu|\be) + (m(\mu) - m(\al))(m(\nu) - m(\be)).
%	\end{aligned}
%	\end{equation*}
%	In particular,
%		$\KL(\mu\otimes\mu|\al\otimes\al) = 2m(\mu)\KL(\mu|\al) + (m(\mu) - m(\al))^2.$
%\end{proposition}
%
%In the Balanced setting, with $(\mu,\nu)$ probabilities, the regularization reads $\KL^\otimes(\pi|\mu\otimes\nu) = 2\KL(\pi|\mu\otimes\nu)$. Thus (up to a factor 2) we retrieve as a particular case the setting of~\cite{peyre2016gromov}.
%
Note that if $\pi\otimes\gamma$ is optimal then so is $(s\pi)\otimes(\frac{1}{s}\gamma)$ with $s\geq 0$. Thus without loss of generality we can optimize under the constraint $m(\pi)=m(\gamma)$ by setting $s=\sqrt{m(\ga) / m(\pi)}$. 

%Such assumption allows to prove that the bi-convex relaxation is tight. We \red{first} provide \redrm{in Appendix}%~\ref{appendix-algo}}
%
% a general result holding for a wide range of quadratic programs, \redrm{but} \red{then} we state here its application for our setting \redrm{in the setting of $\UGW_\epsilon$}.
We now discuss the tightness of the bi-convex relaxation by generalizing a result of Konno. We first present a result which applies to general quadratic assignment problems, then state its application to our setting.
\begin{theorem}[Tight relaxation]
	\label{ThKonno'sGeneralization}
	Let $B$ a Banach space, let $f:B \mapsto \RR \cup \{ +\infty\}$ be a function and
	let $\Ll: C \subset B \mapsto \RR$ the function defined on the convex set $C \subset B$ by $\Ll(\pi) = \frac 12  \langle \pi,k(\pi)\rangle + 2f(\pi)$
	where $k$ is a symmetric bilinear map which is negative (not necessarily definite) on $\Delta C \eqdef \operatorname{Span}(\{ \pi- \ga \,;\, (\pi,\ga) \in C\})$, that is, for any $z \in \Delta C$,
	$\langle z , kz \rangle \leq 0$.
	Assume that there exists $\pi_0 \in C$ such that $\Ll(\pi_0) < +\infty$, and define $\Ff(\pi,\ga) \eqdef \frac 12 \langle \pi,k(\ga)\rangle + f(\pi) + f(\ga)$.
	Then, for any $(\pi_*,\ga_*) \in \arg \min \Ff(\pi,\ga)$, we have $\Ff(\pi_*,\pi_*) = \Ff(\ga_*,\ga_*) = \Ff(\pi_*,\ga_*)$.
	Moreover, if one assumes either that $k$ is a definite kernel or $f$ is strictly convex, one gets $\pi_* = \ga_*$.
\end{theorem}
The above Theorem~\ref{ThKonno'sGeneralization} is proved in Appendix~\ref{appendix-algo}.
As an application, we now state our tightness result for $\GW_\epsilon$ and CGW.
In those settings the optimizers of the bi-convex relaxation are also optimal for the original problem.
\begin{theorem}%[Tightness for $\GW_\epsilon$ and $\CGW$]
	\label{ThTightLowerBound}
For $GW_\epsilon$ with $\epsilon\geq 0$ or for $\CGW$, assume that $\C(t)=t^2$ and that $(d_X,d_Y)$ are both conditionnally negative (or conditionally positive) kernels. Then the bi-convex relaxation of both problems is tight.
\end{theorem}
\begin{proof}
	The proof for $\CGW$ is detailed in Appendix~\ref{sec-app-xp}, we prove the tightness for $\GW_\epsilon$.
	When $\C(t)=t^2$ the kernel $k=\C(|d_X - d_Y|)$ is conditionally negative on the set $\{(\pi,\ga),\,\, \pi_1=\ga_1\;\textrm{and}\;\pi_2=\ga_2\}$, i.e. we have $\dotp{(\pi-\ga)}{k(\pi-\ga)}\leq 0$ (see~\citet{maron2018probably}).
	For $\GW_\epsilon$ one has $\pi_1=\ga_1=\mu$ and $\pi_2=\ga_2=\nu$ thanks to the constraints on marginals. Thus the kernel is negative semi-definite and the proof of Theorem~\ref{ThKonno'sGeneralization} applies, hence the tightness of the relaxation.
\end{proof}
%
%The proof of Theorem~\ref{ThTightLowerBound} is in Appendix~\ref{appendix-algo}. 
%%
%It is based on a generalization of Konno's result~\citep{konno1976maximization} combined with a decomposition formula of $\KL^\otimes$. 
%
Konno's result~\citet{konno1976maximization} applies for unregularized ($\epsilon=0$), Balanced-$\GW$. The novelty of Theorem~\ref{ThTightLowerBound} is its extension to both $\GW_\epsilon$ and $\CGW$.
So far it is an open question whether the relaxation is tight or not for $\UGW_\epsilon$, because the above proof no longer holds.
Note that in all our numerical simulations, our solvers always found solutions of $\UGW_\epsilon$ such that $\pi=\ga$ when $|d_X - d_Y|^2$ is conditonally negative.
The property that the kernel $|d_X- d_Y|^2$ is negative does not hold in general (e.g. for graph geodesic distances) and the tightness of the relaxation remains open in this setting.
We know from~\cite[Theorem 1]{maron2018probably} that it is conditionally negative semi-definite when both $(d_X,d_Y)$ are conditionally negative kernels.
%
%Such property does not hold in general (e.g. for graph geodesic distances) and the tightness of the relaxation remains open in this setting.
%
Examples of distances which are negative kernels are tree metrics in the case of graphs, as well as Euclidean, spherical and hyperbolic distances over their respective manifolds~\cite{feragen2015geodesic}.
In practice, when $\epsilon$ is small, we observed in the indefinite setting that the relaxation outputs more frequently spurious minima than in the negative semi-definite setting.

%In general, this bound is not expected to be tight, but empirically, alternate minimization often converges to a solution with $\pi=\ga$ (as already observed for instance in~\citet{rangarajan1999convergence,solomon2016entropic}), so that the algorithm also finds a local minimizer of the $\UGW_\epsilon$ problem. 
%\redrm{In the Balanced-GW case, this can be explained by the fact that this scheme is equivalent to a mirror descent algorithm~\citep{solomon2016entropic}
%\blue{In the Balanced-GW case in Euclidean spaces, the optimum is known to satisfy $\pi=\ga$~\citep{konno1976maximization} and alternate descent is equivalent to a mirror descent algorithm~\citep{solomon2016entropic}.}
%

\subsection{Alternate Sinkhorn minimization}

Minimizing the lower bound~\eqref{eq-lower-bound} with respect to either $\pi$ or $\ga$ is non-trivial for an arbitrary $\phi$. We restrict our attention to the Kullback-Leibler case $\D_\phi=\rho\KL$ with $\rho >0$, which can be addressed by solving a regularized and convex unbalanced problem as studied in~\citet{chizat2016scaling, sejourne2019sinkhorn}. It is explained in the following proposition.

\begin{proposition}\label{prop-alternate-simple}
	For a fixed $\ga$, the optimal
		$\pi\in\arg\umin{\pi} \Ff(\pi,\ga) +\epsilon\KL(\pi\otimes\gamma|(\mu\otimes\nu)^{\otimes 2})$
	%is the solution of
	solves
	\begin{align*}
		\umin{\pi} &\int c^\epsilon_\ga(x,y) \d\pi(x,y) + \rho m(\ga) \KL(\pi_1|\mu) + \rho m(\ga) \KL(\pi_2|\nu) + \epsilon m(\ga) \KL(\pi|\mu\otimes\nu),
	\end{align*}
	where $m(\ga) \eqdef \ga(X \times Y)$ is the mass of $\ga$, and
	where we define the cost associated to $\ga$ as
	\begin{align*}
		c^\epsilon_\ga(x,y) \!  \eqdef \!  \int \! \C(|d_X(x,\cdot) - d_Y(y,\cdot)|)\d\ga \! + \!  \rho\!\!  \int \! \log(\frac{\d\ga_1}{\d\mu})\d\ga_1 \! 
		+ \! \rho \! \!  \int \! \log(\frac{\d\ga_2}{\d\nu})\d\ga_2 \! 
		 + \!  \epsilon \! \! \int \!   \log(\frac{\d\ga}{\d\mu\d\nu})\d\ga.
	\end{align*}
\end{proposition}

\begin{wrapfigure}{r}{0.44\textwidth}
	\vskip-2.0em
	\begin{minipage}{0.44\textwidth}
		\begin{algorithm}[H]
			\caption{-- \textbf{UGW($\Xx$, $\Yy$, $\rho$, $\epsilon$)} \label{algo-flb-sinkhorn}}
			\small{
				\textbf{Input:} mm-spaces $(\Xx,\Yy)$,  relax. $\rho$, regul. $\epsilon$\\
				\textbf{Output:}~$\pi,\ga$~solving~\eqref{eq-lower-bound}\\
				\vspace*{-1.0em}
				\begin{algorithmic}
					\STATE Init. $\pi=\ga=\mu\otimes\nu / \sqrt{m(\mu) m(\nu)}$, $\g=0$.
					\WHILE{$(\pi,\ga)$ has not converged}
					\STATE Update $\pi\leftarrow\ga$,
					\STATE then  $c \leftarrow c^\epsilon_{\pi}$,   $\;\tilde{\rho} \leftarrow m(\pi)\rho$,   $\;\tilde{\epsilon} \leftarrow m(\pi)\epsilon$
					\vspace*{0.05cm}
					\WHILE{$(\f,\g)$ has not converged}
					\STATE  $ \f \leftarrow -\frac{\tilde{\epsilon}\tilde{\rho}}{\tilde{\epsilon} + \tilde{\rho}}
					\log \int e^{(\g(y) - c(\cdot,y)) / \tilde{\epsilon}}\d\nu(y)$
					\STATE  $\g \leftarrow -\frac{\tilde{\epsilon}\tilde{\rho}}{\tilde{\epsilon} + \tilde{\rho}}
					\log \int e^{(\f(x) - c(x,\cdot)) / \tilde{\epsilon}}\d\mu(x)$
					\ENDWHILE
					\STATE Upd. $\ga(x,y)\!\!\leftarrow\!\! e^{\frac{\f(x)+\g(y)-c(x,y)}{\tilde{\epsilon}}}\mu(x)\nu(y)$
					\STATE Rescale $\ga\leftarrow \sqrt{m(\pi) / m(\ga)} \ga$
					\ENDWHILE
					\STATE Return $(\pi,\ga)$.
				\end{algorithmic}
			}
		\end{algorithm}
	\end{minipage}\vskip-1.5em
\end{wrapfigure}
Computing the cost $c^\epsilon_\ga$ for spaces $X$ and $Y$ of $n$ points has in general a cost $O(n^4)$ in time and memory. However, as explained for instance in~\citet{peyre2016gromov}, for the special case $\C(t)=t^2$, this cost is reduced to $O(n^3)$ in time and $O(n^2)$ in memory. 
%
%\redrm{which allows to scale the method to larger problems.}
%
This is the setting we consider in the numerical simulations.
This makes the method applicable for scales of the order of $10^4$ points. For larger datasets one should use approximation schemes such as hierarchical approaches~\citep{xu2019scalable} or Nystr\"om compression of the kernel~\citep{altschuler2018massively}.

The resulting alternate minimization method is detailed in Algorithm~\ref{algo-flb-sinkhorn}, see Appendix~\ref{appendix-algo} for a discretized version.
It uses the unbalanced Sinkhorn algorithm of~\citet{chizat2016scaling, sejourne2019sinkhorn} as sub-iterations and takes $\pi = \mu\otimes\nu / \sqrt{m(\mu) m(\nu)}$ to initialize the updates.
This Sinkhorn algorithm operates over a pair of continuous functions (so-called Kantorovitch potentials) $f(x)$ and $g(y)$.
For discrete spaces $X$ and $Y$ of size $n$, these functions are stored in vectors of size $n$, and that integral involved in the updates becomes a sum. Each iteration of Sinkhorn thus has a cost $n^2$, and all the involved operation can be efficiently mapped to parallelizable GPU routines as detailed in~\citet{chizat2016scaling, sejourne2019sinkhorn}.
Another advantage of using an unbalanced Sinkhorn algorithm is its complexity $O(n^2 / \epsilon)$ to compute an $\epsilon$-approximation, as stated in~\citet{pham2020unbalanced}, which should be compared to $O(n^2 / \epsilon^2)$ operations for balanced Sinkhorn.

\if 0
\begin{algorithm}[!h]
	\caption{-- \textbf{UGW($\Xx$, $\Yy$, $\rho$, $\epsilon$)} \label{algo-flb-sinkhorn}}
	\textbf{Input:} mm-spaces $(\Xx,\Yy)$,  relaxation $\rho$, regularization $\epsilon$\\
	\textbf{Output:} plans $(\pi,\ga)$ minimizing~\ref{eq-lower-bound}
	\begin{algorithmic}
		\STATE Initialize $\pi=\ga=\mu\otimes\nu / \sqrt{m(\mu) m(\nu)}$, $\g=0$.
		\WHILE{$(\pi,\ga)$ has not converged}
		\STATE Update $\pi\leftarrow\ga$,
		\STATE then  $c \leftarrow c^\epsilon_{\pi}$,   $\;\tilde{\rho} \leftarrow m(\pi)\rho$,   $\;\tilde{\epsilon} \leftarrow m(\pi)\epsilon$
		\vspace*{0.05cm}
		\WHILE{$(\f,\g)$ has not converged}
		\STATE  $\forall x, \; \f(x)\leftarrow -\frac{\tilde{\epsilon}\tilde{\rho}}{\tilde{\epsilon} + \tilde{\rho}}
		\log\Big(\int e^{(\g(y) - c(x,y)) / \tilde{\epsilon}}\d\nu(y) \Big)$
		\STATE  $\forall y, \; \g(y)\leftarrow -\frac{\tilde{\epsilon}\tilde{\rho}}{\tilde{\epsilon} + \tilde{\rho}}
		\log\Big(\int e^{(\f(x) - c(x,y)) / \tilde{\epsilon}}\d\mu(x) \Big)$
		\ENDWHILE
		\STATE Update $\ga(x,y)\leftarrow e^{(\f(x)+\g(y)-c(x,y)) / \tilde{\epsilon}}\mu(x)\nu(y)$
		\STATE Rescale $\ga\leftarrow \sqrt{m(\pi) / m(\ga)} \ga$
		\ENDWHILE
		\STATE Return $(\pi,\ga)$.
	\end{algorithmic}
\end{algorithm}
\fi

Note also that balanced $\GW$ is recovered as a special case when setting $\rho\rightarrow+\infty$, so that $\tilde{\rho} / (\tilde{\epsilon} + \tilde{\rho})\rightarrow 1$ should be used in the iterations.
In order to speed up Sinkhorn inner-loops, especially for small values of $\epsilon$, one can use linear extrapolation~\citep{thibault2017overrelaxed} or non-linear Anderson acceleration~\citep{anderson1965iterative, scieur2016regularized}.

There is an extra scaling step after computing $\ga$ involving the mass $m(\pi)$. It corresponds to the scaling $s$ of $\pi\otimes\gamma$ such that $m(\pi)=m(\gamma)$, and we observe that this scaling is key not only to impose this mass equality but also to stabilize the algorithm. Otherwise we observed that $m(\ga)<1<m(\pi)$ and underflows whenever $m(\ga)\rightarrow 0$ and $m(\pi)\rightarrow\infty$.

% !TEX root = ../neurips_2021.tex

\section{Numerical experiments}
\label{sec-xp}

This section presents simulations on synthetic examples to highlight the qualitative behavior of $\UGW$ and the tightness of the bound $\UGW\geq\CGW$. 
Other illustrations on $\UGW$ are available in Appendix~\ref{sec-app-xp}.
We end the section with a learning application of $\UGW$ in a positive-unlabeled setting, using domain adaptation data so as to compare with PGW~\citet{chapel2020partial}.
In the synthetic experiments, $\mu$ and $\nu$ are probability distributions, which allows us to compare $\GW$ with $\UGW$.

\paragraph{Robustness to imbalanced classes.}

In this first example, we take $X=\RR^3$, $Y=\RR^2$ and consider $\Ee_2$, $\Ee_3$, $\Cc$ and $\Ss$ to be uniform distributions on a 2D and 3D ellipse, a square and a sphere. 
We consider mm-spaces of different dimensions to emphasize the ability of (U)GW to compare different spaces.
Figure~\ref{fig-weight} contrasts the transportation plan obtained by $\GW$ and $\UGW$ for a fixed $\mu=0.5 \Ee_3 + 0.5 \Ss$ and $\nu$ obtained using two different mixtures of $\Ee_2$ and $\Cc$. 
The black segments show the largest entries of the transportation matrix $\pi$, for a sub-sampled set of points (to ease visibility), thus effectively displaying the matching induced by the plan.
Furthermore, the width of the dots are scaled according to the mass of the marginals $\pi_1 \approx \mu$ and $\pi_2 \approx \nu$, i.e. the smaller the point, the smaller is the amount of transported mass.
This figure shows that the exact conservation of mass imposed by $\GW$ leads to a poor geometrical matching of the shapes which have different global mass.
As this should be expected, $\UGW$ recovers coherent matchings. We suspect the alternate minimization algorithm is able to find the global minimum in these cases. 

\begin{figure}[h]
	\centering
	\begin{tabular}{c@{\hspace{1mm}}|@{\hspace{1mm}}c@{\hspace{8mm}}c@{\hspace{1mm}}|@{\hspace{1mm}}c}%
		{\includegraphics[height=.22\linewidth]{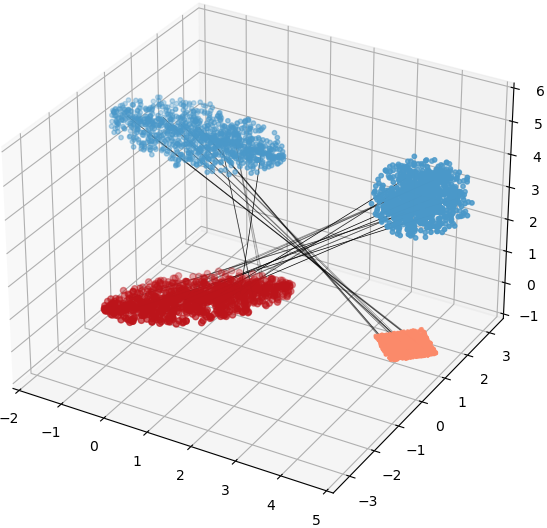}} & 
		{\includegraphics[height=.22\linewidth]{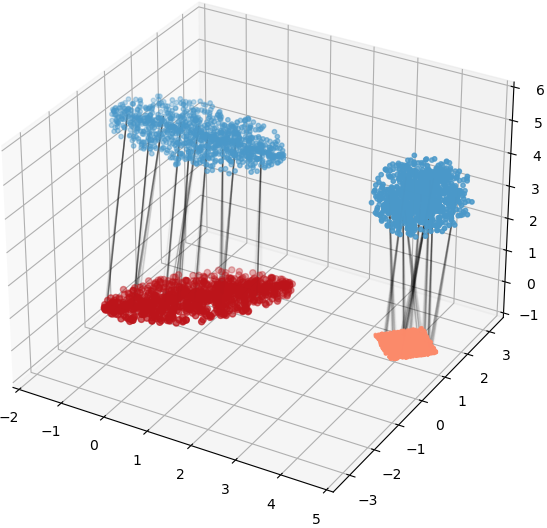}} &
		{\includegraphics[height=.22\linewidth]{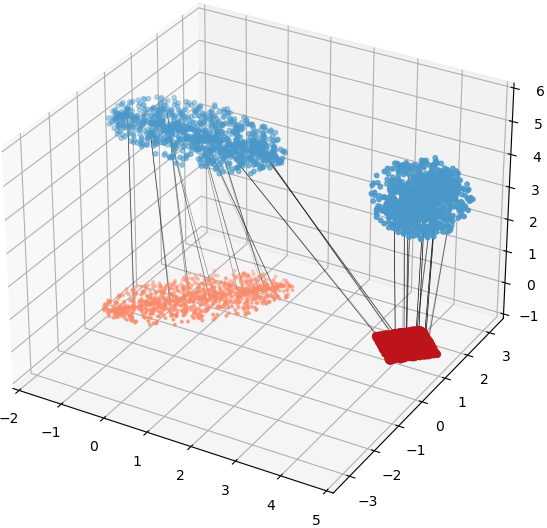}} & 
		{\includegraphics[height=.22\linewidth]{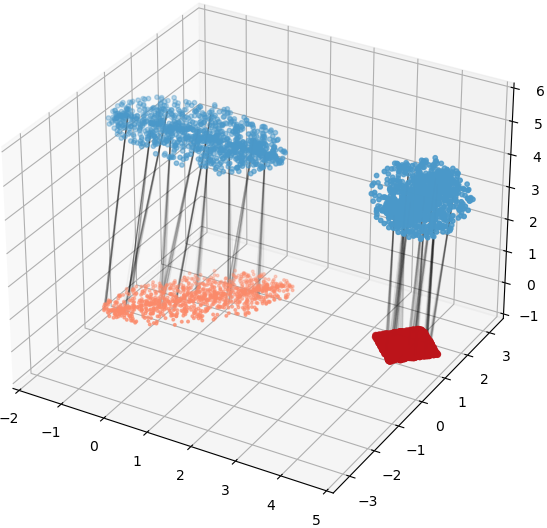}}\\[1mm]
		$\GW$ & $\UGW$ & $\GW$ & $\UGW$
	\end{tabular}
	\caption{$\GW$ vs. $\UGW$ transportation plan, using $\nu=0.3 \Ee_2 + 0.7 \Cc$ on the left, and
		$\nu=0.7 \Ee_2 + 0.3 \Cc$ on the right. The 2D mm-spaces is lifted into $\RR^3$ by padding the third coordinate to zero.}
	\label{fig-weight}
\end{figure}

\paragraph{Tightness of the bound CGW$\leq$UGW}

We propose to approximate $\CGW$ by doing a similar alternate minimization as for $\UGW$, as detailed in Appendix~\ref{sec-app-xp}. 
This numerical scheme does not scale to large problems, but allows us to explore numerically how tight is the upper bound $\UGW\geq\CGW$.
%
% For given mm-spaces, we initialize UGW with enough permutation plans and random plans, such that we can guarantee that we obtain the global minimum.
%
% Thus if we obtain a local minimum of $\CGW$, we prove that the bound is not tight.
%
Figure~\ref{fig:cgw-ugw-local} highlights the fact that in Euclidean space $X=Y=\RR^d$, this bound seems to be tight when the two measures are sufficiently close. 
We consider discrete measures $\mu = \frac{1}{n}\sum_i \de_{x_i}$ in $X=Y=\RR^d$ and $\nu_t = \frac{1}{n}\sum_i \de_{y_i}$ where $y_i=x_i + t \De_i$ where $\De_i$ are random perturbations and denote $(\Xx,\Yy_t)$ the two mm-spaces associated to the Euclidean distance. As $t \rightarrow 0$, $\mu$ and $\nu_t$ get closer, we observe numerically that $\UGW\approx\CGW$. 
%
% We see that as $\Xx$ gets further from $\Yy_t$ the gap between $\UGW$ and $\CGW$ widens.
%
Figure~\ref{fig:cgw-ugw-hist} considers random points $(x_i)_i$ and $(y_i)_i$ and displays the histograms of the ratio $\CGW/\UGW$ for $n=3$. This shows that while the bound $\CGW\leq\UGW$ seems not tight, the ratio appears to be bounded even for points not being close. 
This numerical experiment suggests that $\UGW$ and $\CGW$ are locally equivalent and that $\UGW$ is in practice an acceptable proxy of the distance $\CGW$.
We leave for future works a tighter analysis of the gap between $\UGW$ and $\CGW$.

%\begin{figure}
%	\begin{subfigure}[h]{0.35\linewidth}
%		\includegraphics[width=\linewidth]{sections/figures/fig_compare_locally_ugw_cgw.png}
%%		\caption{Image A}
%	\end{subfigure}
%	\hfill
%%	\hspace*{2cm}
%	\begin{subfigure}[h]{0.35\linewidth}
%		\includegraphics[width=0.95\linewidth]{sections/figures/fig_compare_hist_ugw_cgw.png}
%%		\caption{Image B}
%	\end{subfigure}%
%	\caption{(Left) Comparison of $\UGW(\Xx,\Yy_t)$ and $\CGW(\Xx,\Yy_t)$ as the support gets shifted by a perturbation. (Right) Histograms of the ratio CGW / UGW for $\rho\in\{10^{-1}, 10^{0}, 10^{1}\}$ and random mm-spaces. }
%	\label{fig:cgw-ugw}
%\end{figure}

\begin{wrapfigure}{r}{0.35\textwidth}
	\vskip-.4cm
	\begin{minipage}{0.35\textwidth}
		%		\begin{figure}
		\includegraphics[width=0.95\linewidth]{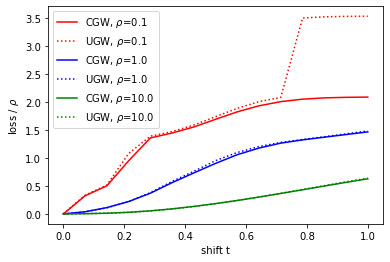}
		\caption{Comparison of $\UGW(\Xx,\Yy_t)$ and $\CGW(\Xx,\Yy_t)$ as the support gets shifted by a perturbation. }
		\label{fig:cgw-ugw-local}
		%		\end{figure}
	\end{minipage}\vskip-.2cm
\end{wrapfigure}

\paragraph{Positive unlabeled learning experiments}

Positive Unlabeled (PU) learning is a semi-supervised classification problem, where instead of learning from positive and negative samples $(x_i,\ell_i)_i$ with labels $\ell_i \in \{-1,1\}$ we only learn from one class labeled with positives, i.e. only those $X \eqdef \enscond{ x_i }{ \ell_i=1 }$.
The task is to leverage $X$ to predict the classes $\ell=\ell(y) \in \{-1,+1\}$ of unlabelled $y \in Y$ belong to a separate space. 
We consider here that $X,Y$ are embedded in Euclidean space, and denote $\Xx,\Yy$ the associated labelled and unlabelled mm-spaces, equipped with the uniform distribution.  
Our experiments are adapted from Partial-GW (PGW)~\citet{chapel2020partial}, which used partial GW to solve PU-learning. 
The rationale of using unbalanced OT methods for PU learning stems from the fact that positive samples should be matched with positive due to their similar features, while negative samples would be ignored due to dissimilar features that induce a laziness to transport mass and match them.

We consider PU learning over the Caltech office dataset used for domain adaptation tasks (with domains Caltech (C)~\citet{griffin2007caltech}, Amazon (A), Webcam (W) and DSLR (D)~\citet{saenko2010adapting}). The Caltech datasets are represented with two embeddings based on Surf and Decaf features~\citet{saenko2010adapting,donahue2014decaf}. 
On the latter datasets, we perform PU learning over similar features (e.g. surf-C $\rightarrow$ surf-* or decaf-C $\rightarrow$ decaf-*) and from one feature format to the other (e.g. surf-C $\rightarrow$ decaf-* or surf-C $\rightarrow$ decaf-*). 
Those features are projected via PCA to subspaces of dimension 10 for surf features and 40 for decaf features.
In the last task, one cannot use standard PU-method, and to the best of our knowledge, Unbalanced-GW methods are the only approaches for PU learning across different domains/features.

\begin{wrapfigure}{r}{0.38\textwidth}
	\vskip-0.4cm
	\begin{minipage}{0.38\textwidth}
		%		\begin{figure}
		\includegraphics[width=0.95\linewidth]{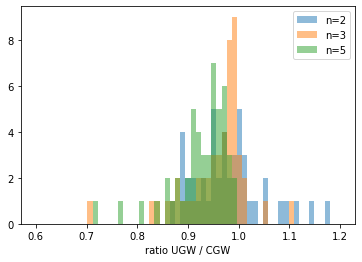}
		\caption{ Histograms of the ratio $\CGW / \UGW$ for random spaces with $n\in\{2,3,5\}$ samples. Ratios over 1 are due to local minima.}
		\label{fig:cgw-ugw-hist}
		%		\end{figure}
	\end{minipage}\vskip-.2cm
\end{wrapfigure}

The procedure is the following. 
% Given features $(X_p, X_u)$, we represent both datasets as mm-spaces $(\Xx_p, \Xx_u)$ endowed with their associated euclidean metric and uniform weights.
%
We solve the PU learning problem by computing the optimal plan $\pi$ for $\UGW(\Xx, \Yy)$. 
We compute its first marginal $\pi_2$ on $\Yy$, and predict the labels of some $y \in Y$ as $\ell(y) \triangleq \sign( \pi_2(y)-q )$ where $q$ is the quantile of $\pi_2$ corresponding to the proportion $r$ of positives samples in $Y$. 
Following~\citet{chapel2020partial} which is adapted from~\citet{kato2018learning, hsieh2019classification}, this proportion $r$ is assumed to be known.
We report the accuracy of the prediction over the same 20 folds of the datasets, and use 20 other folds to validate the parameters of UGW. 
We consider 100 random samples for each fold of $(X,Y)$, a ratio of positive samples $r=0.1$ for domains (C,A,W,D), and a ratio $r=0.2$ for domains (C,A,W).

Since the GW objective is non-convex, the initialization of the minimization algorithms is key to obtain good performances. 
In~\citet{chapel2020partial} and our experiments, for tasks where $Y$ and $Y$ belong to the same Euclidean space, (e.g. surf-C $\rightarrow$ surf-*) we initialize $\pi$ with the Partial-Wasserstein (PW) solution with a squared Euclidean cost.
For cross-domain prediction (e.g. surf-C $\rightarrow$ decaf-*), following~\citet{chapel2020partial}, PGW is initialized with a list of plans built using a coarsened representation of the data with $k$-NN. 
While \citet{chapel2020partial} makes use in an oracle manner of the plan providing the best accuracy, we modified their protocol and keep the plan which has the lowest PGW cost, which seems fairer, hence the difference in performance with~\citet{chapel2020partial}. 
To initialize UGW when $X$ and $Y$ do not belong to the same Euclidean space, we use a UOT solution of a matching between distance histograms called FLB~\citet{memoli2011gromov}.
 We define $\FLB$ in our $\UGW$ setting as
\begin{equation}\label{eq-def-flb}
\FLB(\Xx,\Yy) \eqdef \min \int_{X\times Y} |\bar{\mu}\star d_X - \bar{\nu}\star d_Y|^2 \d\pi + \rho\KL(\pi_1|\mu) + \rho\KL(\pi_2|\nu) + \epsilon\KL(\pi|\mu\otimes\nu),
\end{equation}
where $\mu\star d_X(x) \eqdef \int d_X(x,x')\d\mu(x')$ is the eccentricity, i.e. a histogram of aggregated distances, and $\bar{\mu} = \mu / m(\mu)$. Contrary to GW~\citet{memoli2011gromov}, there is a priori no link between FLB and UGW.

In the experiments we slightly generalize UGW and use two different marginal penalties $\rho_1 \KL^\otimes(\pi_1|\mu) + \rho_2 \KL^\otimes(\pi_2|\nu)$  with two parameters $(\rho_1,\rho_2)$ to take into account shifts between domains/features.
Note that PGW has a single parameter (which plays a role similar to $(\rho_1,\rho_2)$) which controls the cost of mass creation/destruction. 
%
% In this task penalizing $\KL$ with two parameters is close to normalize the mm-spaces.
%
We set $\epsilon=2^{-9}$, which avoids introducing an extra parameter in the method. 
% cross-validating it improves performance, but we set it to reduce over-parameterization, and we observe frequently that it performs best to reduce entropy while guaranteeing numerical stability.
%
The value $(\rho_1,\rho_2)\in\{2^{-k},\; k\in\llbracket5,10\rrbracket\}^2$ are cross validated for each task on the validation folds, and we report the average accuracy on the testing folds.
We discuss in Appendix~\ref{sec-app-xp} the impact of reducing the number of parameters on the performance.
Comparison with other methods -- PU and PUSB~\citet{kato2018learning, du2014analysis} -- are provided in~\citet{chapel2020partial} and we focus here on the comparison with PGW only. 
% since PGW so as to give a proof that the performance of UGW is at least at par with PGW, and performs better in some settings.

The results are reported in Table~\ref{table:data-perf}.
We display the performance of PGW, UGW and the initialization used for UGW to guarantee that using UGW does improve the performance.
We observe that when the source and target dataset is the same (C$\rightarrow$C tasks), the PW initialization performs better and PGW/UGW degrade the performance, so that in this setting Optimal Transport should be preferred over GW, which is to be expected.
However when the domains are different, applying UGW improves the performance over the initialization (which is FLB) in almost all tasks. 
%, thus accounting for more complex shifts than translations. --> pas compris
%
% When the domains are different, the PW initialization is not defined and we use FLB~\eqref{eq-def-flb} to initialize UGW. --> déjà dit ??
Note that in that case the methods PU, PUSB or PW cannot be used. 
%
% The methods PGW and UGW are the only proposals in those settings.
%
Overall, this shows that GW methods are able to solve to some extent the PU learning problem across different spaces, and that using a ``softer'' KL penalties in UGW is at least at par with Partial GW, and performs better in some settings.

\begin{table}[]
	\begin{center}
		\resizebox{\textwidth}{!}{
			\begin{tabular}{|c|c|c|c|c||c|c|c|c|c|}
				\hline
				Dataset                      & prior & Init (PW) & PGW  & \textbf{UGW}    & Dataset                       & prior & Init (FLB) & PGW  & \textbf{UGW} \\
				\hline
				surf-C $\rightarrow$ surf-C  & 0.1   & \textbf{89.9} & 84.9 & 83.9   & surf-C $\rightarrow$ decaf-C & 0.1   & 85.0 & 85.1 & \textbf{85.6}   \\
				surf-C $\rightarrow$ surf-A  & 0.1   & 81.8 & 82.2 & \textbf{83.5}   & surf-C $\rightarrow$ decaf-A & 0.1   & 84.2 & \textbf{87.1} & 83.6   \\
				surf-C $\rightarrow$ surf-W  & 0.1   & \textbf{81.9} & 81.3 & 80.3   & surf-C $\rightarrow$ decaf-W & 0.1   & 86.2 & \textbf{88.6} & 86.8  \\
				surf-C $\rightarrow$ surf-D  & 0.1   & 80.0 & 81.4 & \textbf{83.2}   & surf-C $\rightarrow$ decaf-D & 0.1   & 84.7 & \textbf{91.1} & 90.7   \\
				\hline
				surf-C $\rightarrow$ surf-C  & 0.2   & \textbf{79.7} & 75.7 & 75.4   & surf-C $\rightarrow$ decaf-C & 0.2   & 74.8 & 75.6 & \textbf{75.9}   \\
				surf-C $\rightarrow$ surf-A  & 0.2   & 65.6 & 66.0 & \textbf{76.4}   & surf-C $\rightarrow$ decaf-A & 0.2   & 76.2 & \textbf{87.9} & 82.4   \\
				surf-C $\rightarrow$ surf-W  & 0.2   & 65.1 & 64.3 & \textbf{67.3}   & surf-C $\rightarrow$ decaf-W & 0.2   & 81.5 & 88.4 & \textbf{89.9}   \\
				\hline
				decaf-C $\rightarrow$ decaf-C & 0.1   & \textbf{93.9} & 83.0 & 86.8   & decaf-C $\rightarrow$ surf-C  & 0.1   & \textbf{81.7} & 81.0 & 81.1   \\
				decaf-C $\rightarrow$ decaf-A & 0.1   & 80.1 & 81.4 & \textbf{85.6}   & decaf-C $\rightarrow$ surf-A  & 0.1   & 80.9 & 81.2 & \textbf{82.4}  \\
				decaf-C $\rightarrow$ decaf-W & 0.1   & 80.1 & 82.7 & \textbf{86.1}   & decaf-C $\rightarrow$ surf-W  & 0.1   & 82.0 & 81.3 & \textbf{83.5}   \\
				decaf-C $\rightarrow$ decaf-D & 0.1   & 80.6 & \textbf{83.8} & 83.4   & decaf-C $\rightarrow$ surf-D  & 0.1   & 80.0 & 80.8 & \textbf{81.5}   \\
				\hline
				decaf-C $\rightarrow$ decaf-C & 0.2   & \textbf{90.6} & 76.7 & 80.5   & decaf-C $\rightarrow$ surf-C  & 0.2   & \textbf{66.6} & 63.7 & 65.2   \\
				decaf-C $\rightarrow$ decaf-A & 0.2   & 62.5 & 68.7 & \textbf{74.7}   & decaf-C $\rightarrow$ surf-A  & 0.2   & 62.9 & 62.4 & \textbf{69.3}   \\
				decaf-C $\rightarrow$ decaf-W & 0.2   & 65.7 & 75.9 & \textbf{79.2}   & decaf-C $\rightarrow$ surf-W  & 0.2   & 65.1 & 61.4 & \textbf{83.3}  \\
				\hline
			\end{tabular}
		}
	\end{center}
	\caption{Accuracy for all tasks. The left block are domain adaptation experiments with similar features, where both PGW and UGW are initialised with PW. The right block are domain adaptation experiments with different features, and the reported init is FLB (see Appendix~\ref{sec-app-xp}) used for UGW.}
	\label {table:data-perf}
\end{table}
% !TEX root = ../neurips_2021.tex

\section{Conclusion and perspectives}

This paper defines two Unbalanced Gromov-Wasserstein formulations: $\CGW$ and $\UGW$. We prove that they are both positive and definite. We provide a scalable, GPU-friendly algorithm to compute $\UGW$ illustrate its applicability in learning tasks, and show that $\CGW$ is a distance between mm-spaces up to isometry. 
These divergences and distances allow for the first time to blend in a seamless way the transportation geometry of $\GW$ with creation and destruction of mass. 
This hybridization is the key to unlock both theoretical and practical issues.
This work opens new questions for future works, for instance removing the bias introduced by the use of entropic regularization, which is important for applications to ML. Note that such a debiasing was successfully applied for Balanced-GW in~\citet{bunne2019learning} and is shown to lead to a valid divergence for balanced OT in~\citet{feydy2019interpolating} and UW in~\citet{sejourne2019sinkhorn}.
The design of efficient numerical solvers for $\CGW$ is also an interesting avenue for future works, as well as the study of its induced topology.
%  On the theoretical side, the geodesic structures induced by unbalanced GW distances and divergences is an important subject of study. 
%
% On the practical side, 

\section*{Acknowledgements}

The works of Thibault S\'ejourn\'e and Gabriel Peyr\'e is supported by the ERC grant NORIA.
The work of G. Peyré was supported in part by the French government under management of Agence Nationale de la Recherche as part of the "Investissements d’avenir" program, reference ANR19-P3IA-0001 (PRAIRIE 3IA Institute).

The authors thank Rémi Flamary for his remarks and advices, as well as Laetitia Chapel for her help to reproduce her experiments.

\bibliography{biblio}

\appendix
\onecolumn
\section{Background on unbalanced optimal transport}
\label{sec-app-background}

Following~\citet{liero2015optimal}, this section reviews and generalizes the homogeneous and conic formulations of unbalanced optimal transport. These three formulations are equal in the convex setting of UOT. Our relaxed divergence UGW and conic distance CGW defined in Section~\ref{sec-distance} build upon those constructions but are not anymore equal due to the non-convexity of GW problems.

\subsection{Homogeneous formulation}

To ease the description of the homogeneous formulation, we develop and refactor the Csiszàr divergence terms of~\eqref{eq-uw} in a form analog to Lemma~\ref{lem-rewrite-ugw}. It reads
\begin{equation}
	\begin{aligned}
	\text{UW}(\mu,\nu)^q = \uinf{\pi \in \Mm(X^2)} &\int L_{\C(d(x,y))}(\f(x), \g(y))\d\pi(x,y)
	+\psi^\prime_\infty(|\mu^\bot| + |\nu^\bot|),
	\end{aligned}
\end{equation}
where $L_{c}(r,s) \eqdef c + r\phi(1/r) + s\phi(1/s)$, $|\mu^\bot|\eqdef\mu^\bot(X)$ and $(\f \eqdef \frac{\d\mu}{\d\pi_1}, \g \eqdef \frac{\d\nu}{\d\pi_2})$ are the densities of the Lebesgue decomposition of $(\mu,\nu)$ with respect to $(\pi_1,\pi_2)$ and
\begin{equation}
	\mu = \f \pi_1 + \mu^\bot \qandq \nu = \g \pi_2 + \nu^\bot.  \label{eq-leb-dens-1}
\end{equation}
Such form is helpful to explicit the terms of pure mass creation/destruction $(|\mu^\bot| + |\nu^\bot|)$ and reinterpret the integral under $\pi$ as a transport term with a new cost $L_{\C(d)}$.

Then the authors of~\citet{liero2015optimal} define the homogeneous formulations HUW as
\begin{equation}
	\begin{aligned}\label{eq-uw-homog}
	\text{HUW}(\mu,\nu)^q \eqdef \uinf{\pi \in \Mm(X^2)} &\int H_{\C(d(x,y))}(\f(x), \g(y))\d\pi(x,y)
	+\psi^\prime_\infty(|\mu^\bot| + |\nu^\bot|),
	\end{aligned}
\end{equation}
 where the 1-homogeneous function $H_c$ is the perspective transform of $L_c$
\begin{align}
	H_c(r, s) \eqdef \inf_{\theta\geq 0} \theta\big( c + \psi(\tfrac{r}{\theta}) + \psi(\tfrac{s}{\theta} ) \big) 
	= \inf_{\theta\geq 0} \theta L_c(\tfrac{r}{\theta}, \tfrac{s}{\theta}).
\end{align}
By definition one has $L_c\geq H_c$, though after optimization one has $\text{\upshape{UW}}=\text{\upshape{HUW}}$.

\subsection{Cone sets, cone distances and explicit settings}
\label{sec-setups}

The conic formulation detailed in Section~\ref{sec-conic-uw} is obtained by performing the optimal transport on the cone set $\Co[X] \eqdef X\times\RR_+ / (X\times\{0\})$, where the extra coordinate accounts for the mass of the particle.
Coordinates of the form $(x,0)$ are merged into a single point called the apex of the cone, noted $\zc_X$. In the sequel, points of $X\times\RR_+$ are noted $(x,r)$ and those of $\Co[X]$ are noted $[x,r]$ to emphasize the quotient operation at the apex.

For a pair $(p,q)\in\RR_+$, we define for any $[x,r],[y,s]\in\Co[X]^2$
\begin{align}
	\Dd_{\Co[X]}([x,r], [y,s])^q \eqdef H_{\C(d(x,y))}(r^p, s^p).
\end{align}
In general $\Dd_{\Co[X]}$ is not a distance, but it is always definite as proved by the following result described in~\cite{de2019metric}.
\begin{proposition}
	Assume that $d$ is definite, $\C^{-1}(\{0\})=\{0\}$ and $\phi^{-1}(\{0\})=\{1\}$. Assume also that for any $(r,s)$, there always exists $\theta^*$ such that $H_c(r, s)= \theta^* L_c(\tfrac{r}{\theta^*}, \tfrac{s}{\theta^*})$.
	Then $\Dd_{\Co[X]}$ is definite on $\Co[X]$, i.e. $\Dd_{\Co[X]}([x,r], [y,s]) =0$ if and only if $(r=s=0) \;\textrm{or}\; (r=s \;\textrm{and}\; x=y)$.
\end{proposition}
\begin{proof}
	Assume $\Dd_{\Co[X]}([x,r], [y,s]) =0$, and write $\theta^*$ such that
	\begin{align*}
		\Dd_{\Co[X]}([x,r], [y,s])^q = \theta^* L_c(\tfrac{r^p}{\theta^*}, \tfrac{s^p}{\theta^*})
		= \theta^* \C(d(x,y)) + r^p\phi(\tfrac{\theta^*}{r^p}) + s\phi(\tfrac{\theta^*}{s^p}),
	\end{align*}
	where the last line is given by the definition of reverse entropy.
	There are two cases. If $\theta^* >0$, since all terms are positive, there are all equal to $0$. By definiteness of $d$ it yields $x=y$ and because $\phi^{-1}(\{0\})=\{1\}$ we have $r^p=s^p=\theta^*$ and $r=s$.
	If $\theta^*=0$ then $\Dd_{\Co[X]}([x,r], [y,s])^q=\phi(0)(r^p+s^p)$. The assumption $\phi^{-1}(\{0\})=\{1\}$ implies $\phi(0)>0$, thus necessarily $r=s=0$.
\end{proof} 
 
 The function $H_c$ can be computed in closed form for a certain number of common entropies $\phi$, and we refer to~\citet[Section 5]{liero2015optimal} for an overview. 
 Of particular interest are those $\phi$ where $\Dd_{\Co[X]}$ is a distance, which necessitates a careful choice of $\C,p$ and $q$. We now detail three particular settings where this is the case. In each setting we provide $(\D_\phi, \C, p, q)$ and its associated cone distance $\Dd_{\Co[X]}$. 

\paragraph{Gaussian Hellinger distance}

It corresponds to
\begin{align*}
	\D_\phi = \KL, \quad \C(t) = t^2 \qandq q=p=2,\\
	\Dd_{\Co[X]}([x,r], [y,s])^2 = r^2 + s^2 - 2rse^{-d(x,y) / 2},
\end{align*}
in which case it is proved in~\citet{liero2015optimal} that $\Dd_{\Co[X]}$ is a cone distance.

\paragraph{Hellinger-Kantorovich / Wasserstein-Fisher-Rao distance}

It reads
\begin{align*}
	\D_\phi =\KL, \quad \C(t) = -\log\cos^2(t\wedge\tfrac{\pi}{2}) \qandq q=p=2,\\
	\Dd_{\Co[X]}([x,r], [y,s])^2 = r^2 +s^2 - 2rs\cos(\tfrac{\pi}{2}\wedge d(x,y)),
\end{align*}
in which case it is proved in~\citet{burago2001course} that $\Dd_{\Co[X]}$ is a cone distance.

The weight $\C(t) = -\log\cos^2(t\wedge\tfrac{\pi}{2})$, which might seem more peculiar, is in fact the penalty that makes unbalanced OT a length space induced by the Gaussian-Hellinger distance (if the ground metric $d$ is itself geodesic), as proved in~\citet{liero2016optimal,chizat2018interpolating}.
This weight introduces a cut-off, because $\C(d(x,y))=+\infty$ if $d(x,y)>\pi/2$. There is no transport between points too far from each other. The choice of $\pi/2$ is arbitrary, and can be modified by scaling $\la \mapsto \la(\cdot/s)$ for some cutoff $s$.

\paragraph{Partial optimal transport}
It corresponds to
\begin{align*}
	\D_\phi = \TV, \quad \C(t)=t^q \qandq q\geq 1 \qandq p=1,\\
	\Dd_{\Co[X]}([x,r], [y,s])^q = r + s - (r\wedge s)(2-d(x,y)^q)_+,
\end{align*}
in which case it is proved in~\citet{chizat2018unbalanced} that $\Dd_{\Co[X]}$ is a cone distance.
The case $\D_\phi=\TV$ is equivalent to partial unbalanced OT, which produces discontinuities (because of the non-smoothness of the divergence) between regions of the supports which are being transported and regions where mass is being destroyed/created.
Note that~\citet{liero2015optimal} do not mention that this $\Dd_{\Co[X]}$ defines a distance, so this result is new to the best of our knowledge, although it can be proved without a conic lifting that partial OT defines a distance as explained in~\citet{chizat2018unbalanced}.

\subsection{Conic formulation of UW}
\label{sec-conic-uw}

The last formulation reinterprets UW as an OT problem on the cone, with the addition of two linear constraints. Informally speaking, $H_c$ becomes $\Dd_{\Co[X]}$, the term $(|\mu^\bot| + |\nu^\bot|)$ is taken into account by the constraints~\eqref{eq-uw-conic-set} below, and the variables $(\f,\g)$ are replaced by $(r^p,s^p)$. It reads
\begin{align}
\text{CUW}(\mu,\nu)^q \eqdef \inf_{\al\in \Uu_p(\mu,\nu)} \int \Dd_{\Co[X]}([x, r], [y, s]))^q\d\al([x,r], [y,s]),
\end{align}
where the constraint set $\Uu_{p}(\mu,\nu)$ is defined as
\begin{align}\label{eq-uw-conic-set}
\Uu_{p}(\mu,\nu)  \eqdef \enscond{ \al\in\Mm_+(\Co[X]^2) }{
	\int_{\RR_+} r^p \d\al_1(\cdot,r)=\mu,
	\int_{\RR_+} s^p \d\al_2(\cdot,s)=\nu
}.
\end{align}
Thus CUW consists in minimizing the Wasserstein distance $\text{W}_{\Dd_{\Co[X]}}(\al_1,\al_2)$ on the cone $(\Co[X], \Dd_{\Co[X]})$. The additional constraints on $(\al_1,\al_2)$ mean that the lift of the mass on the cone must be consistent with the total mass of $(\mu,\nu)$. When $\Dd_{\Co[X]}$ is a distance, CUW inherits the metric properties of $W_{\Dd_{\Co[X]}}$. Our theoretical results rely on an analog construction for GW.

The following proposition states the equality of the three formulations and recapitulates its main properties. The proofs are detailed in~\citet{liero2015optimal}.
\begin{proposition}[From~\citet{liero2015optimal}]
	One has $\text{\upshape{UW}}=\text{\upshape{HUW}}=\text{\upshape{CUW}}$, which are symmetric, positive and definite.
	Furthermore, if $(X,d_X)$ and $(\Co[X], \Dd_{\Co[X]})$ are metric spaces with $X$ separable, then $\Mm_+(X)$ endowed with $\text{\upshape{CUW}}$ is a metric space.
\end{proposition}
\begin{proof}
	The equality $\text{UW}=\text{HUW}$ is given by~\citet[Theorem 5.8]{liero2015optimal}, while the equality $\text{HUW}=\text{CUW}$ holds thanks to~\citet[Theorem 6.7 and Remark 7.5]{liero2015optimal}, where the latter theorem can be straightforwardly generalized to any cone distance built as in Section~\ref{sec-setup-dist}. Since $\Dd_{\Co[X]}$ is symmetric, positive and definite (see Proposition~\ref{prop-cone-dist-definite}), then so is CUW. Furthermore, if $\Dd_{\Co[X]}$ satisfies the triangle inequality, separability of $X$ allows to apply the gluing lemma~\citep[Corollary 7.14]{liero2015optimal} which generalizes to any exponent $p$ defining $\Uu_{p}(\mu,\nu)$ and any cone distance $\Dd_{\Co[X]}$.
\end{proof}

\newpage
\section{UGW formulation and definiteness}
\label{appendix-distance-ugw}
We present in this section the proofs of the properties of our divergence $\UGW$. We refer to Section~\ref{sec-distance} for the definition of the $\UGW$ formulation and its related concepts. For conciseness we write $\Gamma(x,x',y,y') = |d_X(x,x') - d_Y(y,y')|$.

We first start with the existence of minimizers stated in Proposition~\ref{thm-exist-minimizer}. It illustrates in some sense that our divergence is well-defined.
\begin{proposition}[Existence of minimizers]
Assume $(\Xx,\Yy)$ to be compact mm-spaces and that we either have
\begin{enumerate}
  \item $\phi$ superlinear, i.e $\phi^\prime_\infty=\infty$
  \item $\C$ has compact sublevel sets in $\RR_+$ and $2\phi^\prime_\infty + \inf \C >0$
\end{enumerate}
Then there exists $\pi\in\Mm_+(X\times Y)$ such that $\UGW(\Xx,\Yy)=\Ll(\pi)$.
\end{proposition}
\begin{proof}
We adapt here from ~\citet[Theorem 3.3]{liero2015optimal}. The functional is lower semi-continuous as a sum of l.s.c terms.
 Thus it suffices to have relative compactness of the set of minimizers. Under either one of the assumptions, coercivity of the functional holds thanks to Jensen's inequality
\begin{align*}
	\Ll(\pi)&\geq m(\pi)^2\inf\C(\Gamma) +  m(\mu)^2 \phi(\frac{m(\pi)^2}{m(\mu)^2}) +  m(\nu)^2 \phi(\frac{m(\pi)^2}{m(\nu)^2})\\
	&\geq m(\pi)^2 \Big[ \inf\C(\Gamma) +  \frac{m(\mu)^2}{m(\pi)^2} \phi(\frac{m(\pi)^2}{m(\mu)^2}) +  \frac{m(\nu)^2}{m(\pi)^2} \phi(\frac{m(\pi)^2}{m(\nu)^2})\Big].
\end{align*}
As $m(\pi)\rightarrow +\infty$ the right hand side converges to $2\phi^\prime_\infty + \inf \C >0$, which under either one of the assumptions yields $\Ll(\pi)\rightarrow +\infty$, hence the coercivity.
Thus we can assume there exists some $M$ such that $m(\pi)<M$. Since the spaces are assumed to be compact, the Banach-Alaoglu theorem holds and gives relative compactness in $\Mm_+(X\times Y)$.

Take any sequence of plans $\pi_n$ that approaches $\UGW(\Xx,\Yy)=\inf \Ll(\pi)$. Compactness gives that a subsequence $\pi_{n_k}$ weak* converges to some $\pi^*$. Because $\Ll$ is l.s.c, we have $\Ll(\pi^*) \leq \inf \Ll(\pi)$, thus $\Ll(\pi^*) = \inf \Ll(\pi)$. The existence of such limit reaching the infimum gives the existence of a minimizer.
\end{proof}

Note that this formulation is nonegative and symmetric because the functional $\Ll$ is also nonegative and symmetric in its inputs $(\Xx,\Yy)$. This formulation allows straightforwardly to prove the definiteness of $\UGW$.

\begin{proposition}[Definiteness of $\UGW$]\label{thm-ugw-definite}
Assume that $\phi^{-1}(\{0\})=\{1\}$ and $\C^{-1}(\{0\})=\{0\}$.
The following assertions are equivalent:
\begin{enumerate}
  \item $\UGW(\Xx,\Yy)=0$
  \item $\exists\pi\in\Mm_+(X\times Y)$ whose marginals are $(\mu,\nu)$ such that $d_X(x, x')= d_Y(y, y')$ for $\pi\otimes\pi$-a.e. $(x, x', y, y')\in (X\times Y)^2$.
  \item There exists a mm-space $(Z, d_Z, \eta)$ with full support and Borel maps $\psi_X:Z\rightarrow X$ and $\psi_Y:Z\rightarrow Y$. such that $(\psi_X)_\sharp \eta =\mu$, $(\psi_Y)_\sharp \eta =\nu$ and $d_Z = (\psi_X)^\sharp d_X = (\psi_Y)^\sharp d_Y$
  \item There exists a Borel measurable bijection between the measures' supports $\psi:spt(\mu)\rightarrow spt(\nu)$ with Borel measurable inverse such that $\psi_\sharp\mu = \nu$ and $d_Y = \psi^\sharp d_X$.
\end{enumerate}
\end{proposition}
\begin{proof}
Recall that $(2) \Leftrightarrow (3) \Leftrightarrow (4)$ from~\citet[Lemma 1.10]{sturm2012space}. thus it remains to prove $(1)\Leftrightarrow(2)$.

If there is such coupling plan $\pi$ between $(\mu,\nu)$ then one has $\pi\otimes\pi$-a.e. that $\Gamma=0$, and all $\phi$-divergences are zero as well, yielding a distance of zero a.e.

Assume now that $\UGW(\Xx,\Yy)=0$, and write $\pi$ an optimal plan. All terms of $\Ll$ are positive, thus under our assumptions we have $\Gamma=0$, $\pi_1\otimes\pi_1=\mu\otimes\mu$ and $\pi_2\otimes\pi_2=\nu\otimes\nu$. Thus we get that $\pi$ has marginals $(\mu,\nu)$ and that $d_X(x, x')= d_Y(y, y')$ $\pi\otimes\pi$-a.e.
\end{proof}

We end with a result on the reformulation of $\UGW$ which is the first step to connnect it with the conic formulation $\CGW$. It is the same proof as in the main body.
\begin{lemma}
	Defining $L_{c}(r,s) \eqdef c + r\phi(1/r) + s\phi(1/s)$, and writing $(\f \eqdef \frac{\d\mu}{\d\pi_1}, \g \eqdef \frac{\d\nu}{\d\pi_2})$ the Lebesgue densities of $(\mu,\nu)$ w.r.t. $(\pi_1,\pi_2)$ such that
	$\mu = \f \pi_1 + \mu^\bot$ and $\nu = \g \pi_2 + \nu^\bot$, one has
	\begin{align*}
	\Ll(\pi)= &\int_{X^2 \times Y^2} L_{\C(\Gamma)}(\f\otimes\f,\g\otimes\g)\d\pi\d\pi+\phi(0)(|(\mu\otimes\mu)^\bot| + |(\nu\otimes\nu)^\bot|).
	\end{align*}
\end{lemma}
\begin{proof}
	Using Equation~\eqref{eq-tensor-leb-dens}, one has
	\begin{align*}
	\Ll(\pi) &= \int_{X^2 \times Y^2} \C(\Gamma)\d\pi\d\pi
	+ \D_\phi^\otimes(\pi_1|\mu) + \D_\phi^\otimes(\pi_2|\nu)\\
	&= \int_{X^2 \times Y^2} \C(\Gamma)\d\pi\d\pi
	+ \D_\psi^\otimes(\mu|\pi_1) + \D_\psi^\otimes(\nu|\pi_2)\\
	&= \int_{X^2 \times Y^2} \C(\Gamma)\d\pi\d\pi + \int_{X^2}\psi(\f\otimes\f)\d\pi_1\d\pi_1 + \int_{Y^2}\psi(\g\otimes\g)\d\pi_2\d\pi_2\nonumber\\
	&\qquad+\phi(0)(|(\mu\otimes\mu)^\bot| + |(\nu\otimes\nu)^\bot|)\\
	&= \int_{X^2 \times Y^2} L_{\C(\Gamma)}(\f\otimes\f,\g\otimes\g)\d\pi\d\pi+\phi(0)(|(\mu\otimes\mu)^\bot| + |(\nu\otimes\nu)^\bot|).\nonumber
	\end{align*}
\end{proof}

%%%%%%%%%%%%%%%%%%%%%%%%%%%%%%%%%%%%%%%%%%%%%%%%%%%%%%%%%%%%%%%%%%%%%%%%%
\section{Conic formulation and metric properties}
\label{appendix-distance-cgw}

We present in this section the proofs of the properties mentioned in Section~\ref{sec-distance}. We refer to Section~\ref{sec-distance} and Appendix~\ref{sec-app-background} for the definition of the conic formulation and its related concepts.

In this section we frequently use the notion of marginal for neasures. For any sets $E,F$, we write $\margp{E}:E\times F\rightarrow E$ the \textbf{canonical projection} such that for any $(x,y)\in E\times F,\, \margp{E}(x,y)=x$.  Consider two complete separable mm-spaces $\Xx = (X, d_X, \mu)$ and $\Yy=(Y, d_Y, \nu)$. Write $\pi\in\Mm_+(X\times Y)$ a coupling plan, and define its marginals by $\pi_1 = \margp{X}_\sharp\pi$ and $\pi_2 = \margp{Y}_\sharp\pi$. The definition of the marginals can also be seen by the use of test functions. In the case of $\pi_1$ it reads for any test function $\xi$
\begin{align*}
	\int \xi(x)\d\pi_1(x) = \int \xi(x) \d\pi(x,y).
\end{align*}

\subsection{Preliminary results}
We present in this section concepts and properties which are necessary for the proof of Theorem~\ref{thm-ugw-dist}. We introduce a dilation operator whose role is to rescale the radial coordinate of a measure with a given scaling.

\begin{definition}[dilations]
	Consider $v([x, r], [y, s])$ a Borel measurable scaling function depending on $[x, r], [y, s]\in\Co[X]\times\Co[Y]$. Take a plan $\al\in\Mm_+(\Co[X]\times\Co[Y])$. We define the dilation $\dil{v}: \al\mapsto (h_v)_\sharp(v^{p}\al)$ where
	\begin{align*}
		h_v([x, r], [y, s]) \eqdef ([x, r/w], [y, s/w]),
	\end{align*}
	where $w = v([x, r], [y, s])$. It reads for any test function $\xi$
	\begin{align*}
		\int \xi([x, r], [y, s])\d\dil{v}(\al) = \int\xi([x, r/w], [y, s/w])w^p \d\al.
	\end{align*}
\end{definition}

The importance of dilations is given by the following lemma.

\begin{lemma}[Invariance to dilation]
	The problem $\CGW$ is invariant to dilations, i.e. for any $\al\in\Uu_p(\mu,\nu)$, we have $\dil{v}(\al)\in\Uu_p(\mu,\nu)$ and $\Hh(\al) = \Hh(\dil{v}(\al))$.
\end{lemma}
\begin{proof}
	First we prove the stability of $\Uu_p(\mu,\nu)$ under dilations. Take $\al\in\Uu_p(\mu,\nu)$. For any test function $\xi$ defined on $X$ we have
	\begin{align*}
		\int \xi(x)r^p\d\dil{v}(\al) = \int \xi(x)(\frac{r}{v})^p.v^p\d(\al) = \int\xi(x)r^p\d\al = \int\xi(x)\d\mu(x).
	\end{align*}
Similarly we get $\margp{Y}_\sharp(s^q \dil{v}(\al)) = \nu$, thus $\dil{v}(\al)\in\Uu_p(\mu,\nu)$.

It remains to prove the invariance of the functional. Recall that $\Dd^q$ is p-homogeneous. It yields
\begin{align*}
	\Hh(\dil{v}(\al)) &= \int \Dd([d_X(x,x'), rr'], [d_Y(y,y'), ss']))^q\d\dil{v}(\al)\d\dil{v}(\al)\\
	&= \int\Dd([d_X(x,x'), \frac{r}{v}\cdot\frac{r'}{v}], [d_Y(y,y'), \frac{s}{v}\cdot\frac{s'}{v}]))^q v^{p}\cdot v^p\d\al \d\al\\
	&= \int\frac{1}{v^{2p}}\Dd([d_X(x,x'), rr'], [d_Y(y,y'), ss']))^q v^{2p}\d\al \d\al\\
	&= \int\Dd([d_X(x,x'), rr'], [d_Y(y,y'), ss']))^q \d\al \d\al\\
	& = \Hh(\al)
\end{align*}
Both the functional and the constraint set are invariant, thus the whole CGW problem is invariant to dilations.
\end{proof}

The above lemma allows to normalize the plan such that one of its marginal is fixed to some value. Fixing a marginal allows to generalize the gluing lemma which is a key ingredient of the triangle inequality in optimal transport.
\begin{lemma}[Normalization lemma]\label{lem-norm}
	Assume there exists $\al\in\Uu_p(\mu,\nu)$ such that $\CGW(\Xx,\Yy)=\Hh(\al)$. Then there exists $\tilde{\al}$ such that $\tilde{\al}\in\Uu_p(\mu,\nu)$ and $\CGW(\Xx,\Yy)=\Hh(\tilde{\al})$ and whose marginal on $\Co[Y]$ is $\nu_{\Co[Y]}=\margp{\Co[Y]}\sharp\tilde{\al} = \delta_{\zc_Y} + \margc_\sharp(\nu \otimes \delta_1)$, where $\margc$ is the canonical injection from $Y\times\RR_+$ to $\Co[Y]$.
\end{lemma}
\begin{proof}
	The proof is exactly the same as~\citet[Lemma 7.10]{liero2015optimal} and is included for completeness. Take an optimal plan $\al$. Because the functional and the constraints are homogeneous in $(r,s)$, the plan $\hat{\al} = \al + \delta_{\zc_X}\otimes\delta_{\zc_Y}$ verifies $\hat{\al}\in\Uu_p(\mu,\nu)$ and $\Hh(\hat{\al}) = \Hh(\al)$. Indeed, because of this homogeneity the contribution $\delta_{\zc_X}\otimes\delta_{\zc_Y}$ has $(r,s)=(0,0)$ which has thus no impact.
	
	Considering $\hat{\al}$ instead of $\al$ allows to assume without loss of generality that the transport plan charges the apex, i.e. setting
	\begin{align}
	S = \{[x,r],[y,s]\in\Co[X]\times\Co[Y], [y,s]=\zc_Y\},
	\end{align}
	one has $\omega_Y \eqdef \hat{\al}(S) \geq 1$.
Then we can define the following scaling
\begin{align}
	v([x,r], [y,s]) = 
	\begin{cases}
		s \textrm{  if  }s>0\\
		\omega_Y^{-1/q} \textrm{  otherwise}.
	\end{cases}
\end{align}

We prove now that $\dil{v}(\hat{\al})$ has the desired marginal on $\Co(Y)$ by considering test functions $\xi([y,s])$. We separate the integral into two parts with the set $S$, and write $\hat{\al} = \rest{\hat{\al}}{S} + \rest{\hat{\al}}{S^c}$ their restrictions to $S$ and $S^c$ respectively.
It reads
\begin{align*}
	\int \xi([y,s])\d\dil{v}(\hat{\al}) &= \int\xi([y,s / v])v^p\d\hat{\al}\\
	& = \int\xi([y,s / v])v^p\d\rest{\hat{\al}}{S} + \int\xi([y,s / v])v^p\d\rest{\hat{\al}}{S^c}\\
	&=\int\xi(\zc_Y)\omega_Y^{-1}\d\rest{\hat{\al}}{S} + \int\xi([y,s / s])s^p\d\rest{\hat{\al}}{S^c}\\
	& = \xi(\zc_Y)\cdot\omega_Y\cdot\omega_Y^{-1} + \int\xi([y,1])s^p\d\hat{\al}\\
	& = \xi(\zc_Y) + \int\xi(\margc(y,s))\d(\nu(y)\otimes\delta_1(s))\\
	& = \int\xi([y,s])\d(\delta_{\zc_Y} + \margc_\sharp(\nu \otimes \delta_1)),
\end{align*}
which is the formula of the desired marginal on $\Co[Y]$. Since $\hat{\al}\in\Uu_p(\mu,\nu)$, its dilation is also in $\Uu_p(\mu,\nu)$, and $\Hh(\al) = \Hh(\hat{\al})=\Hh(\dil{v}(\hat{\al}))$.
\end{proof}

\subsubsection{Proof of Theorem~\ref{thm-ugw-dist}}

\emph{Non-negativity} and \emph{symmetry} hold since $\Hh$ is a sum of non-negative symmetric terms.
To prove \emph{Definiteness}, assume $\CGW(\Xx,\Yy)=0$, and write $\al$ an optimal plan. We have $\al\otimes\al$-a.e. that $d_X(x,x')=d_Y(y,y')$ and $rr'=ss'$ because $\Dd$ is definite (see Proposition~\ref{prop-cone-dist-definite}). Thanks to the completeness of $(\Xx,\Yy)$ and a result from~\citet[Lemma 1.10]{sturm2012space}, such property implies the existence of a Borel isometric bijection with Borel inverse between the supports of the measures $\psi:\Supp(\mu)\rightarrow\Supp(\nu)$, where $\Supp$ denotes the support. The bijection $\psi$ verifies $d_X(x,x')=d_Y(\psi(x),\psi(x'))$. To prove $\Xx\sim\Yy$ it remains to prove $\psi_\sharp\mu=\nu$. Due to the density of continuous functions of the form $\xi(x)\xi(x')$, the constraints of $\Uu_p(\mu,\nu)$ are equivalent to
\begin{align*}
	\int_{\RR_+} (rr')^p \d\al_1(\cdot,r)\d\al_1(\cdot,r')=\mu\otimes\mu,\quad
	\int_{\RR_+} (ss')^p \d\al_2(\cdot,s)\d\al_2(\cdot,s')=\nu\otimes\nu.
\end{align*}
Take a continuous test function $\xi$ defined on $\Supp(\nu)^2$. Writing $y=\psi(x)$ and $y'=\psi(x')$, one has
\begin{align*}
	\int\xi(y,y')\d\nu\d\nu &= \int\xi(y,y') (ss')^p\d\al\d\al\\
	&= \int\xi(\psi(x),\psi(x')) (ss')^p\d\al\d\al\\
	&= \int\xi(\psi(x),\psi(x')) (rr')^p\d\al\d\al\\
&= \int\xi(\psi(x),\psi(x')) \d\mu\d\mu\\
	&= \int\tilde{\xi}(x,x') \d\psi_\sharp\mu\d\psi_\sharp\mu.
\end{align*}
Since $\psi$ is a bijection, there is a bijection between continuous functions $\xi$ of $\Supp(\nu)^2$ and functions $\tilde{\xi}$ of $\Supp(\mu)^2$. Thus we obtain $\nu=\psi_\sharp\mu$ and we have $\Xx\sim\Yy$.

It remains to prove the \emph{triangle inequality}. Assume now that $\Dd$ satisfies it.
Given three mm-spaces $(\Xx,\Yy,\Zz)$ respectively equipped with measures $(\mu,\nu,\eta)$, consider $\al,\be$ which are optimal plans for $\CGW(\Xx,\Yy)$ and $\CGW(\Yy,\Zz)$. 	
Using Lemma~\ref{lem-norm} to both $\al$ and $\be$, we can consider measures $(\bar{\al},\bar{\be})$ which are also optimal and have a common marginal $\bar\nu$ on $\Co[Y]$. Thanks to this common marginal and the separability of $(X,Y,Z)$, the standard gluing lemma~\citep[Lemma 7.6]{villani2003} applies and yields a glued plan $\ga\in\Mm_+(\Co[X]\times\Co[Y]\times\Co[Z])$ whose respective marginals on $\Co[X]\times\Co[Y]$ and $\Co[Y]\times\Co[Z]$ are $(\bar{\al},\bar{\be})$. Furthermore, the marginal $\bar{\ga}$ of $\ga$ on $\Co[X]\times\Co[Z]$ is in $\Uu_p(\mu,\eta)$. Indeed, $(\bar{\ga},\bar{\al})$ have the same marginal on $\Co[X]$ and same for $(\bar{\ga},\bar{\be})$ on $\Co[Z]$, hence this property.
Write $d_X=d_X(x,x')$ for sake of conciseness (and similarly for $Y,Z$). The calculation reads
\begin{align}
	&\CGW(\Xx, \Zz)^{\tfrac{1}{q}}\\ 
	&\leq \Big(\int \Dd([d_X, rr'],[d_Z, tt'])^q\d\bar{\ga}([x,r],[z,t])\d\bar{\ga}([x',r'],[z',t'])\Big)^{\tfrac{1}{q}}\label{eq-triangle-1}\\
	\quad &\leq\Big(\int \Dd([d_X, rr'],[d_Z, tt'])^q\d\ga([x,r],[y,s],[z,t])\d\ga([x',r'],[y',s'],[z',t'])\Big)^{\tfrac{1}{q}}\label{eq-triangle-2}\\
	\quad &\leq \Big(\int (\Dd([d_X, rr'],[d_Y, ss']) + \Dd([d_Y, ss'],[d_Z, tt']))^q\d\ga\d\ga\Big)^{\tfrac{1}{q}}\label{eq-triangle-3}\\
	\quad &\leq \Big(\int \Dd([d_X, rr'],[d_Y,ss'])^q\d\ga\d\ga\Big)^{\tfrac{1}{q}}  +\Big(\int\Dd([d_Y,ss'],[d_Z,tt'])^q\d\ga\d\ga\Big)^{\tfrac{1}{q}}\label{eq-triangle-4}\\
	&\leq \Big(\int \Dd([d_X, rr'],[d_Y, ss'])^q\d\bar{\al}([x,r],[y,s])\d\bar{\al}([x',r'],[y',s'])\Big)^{\tfrac{1}{q}}\nonumber\\
	&\qquad+ \Big(\int \Dd([d_Y, ss'],[d_Z, tt'])^q\d\bar{\be}([y,s],[z,t])\d\bar{\be}([y',s'],[z',t'])\Big)^{\tfrac{1}{q}}\label{eq-triangle-5}\\
	&\leq \CGW(\Xx, \Yy)^{\tfrac{1}{q}}+ \CGW(\Yy, \Zz)^{\tfrac{1}{q}}.\label{eq-triangle-6}
\end{align}
Since $\bar{\ga}\in\Uu_p(\mu,\eta)$, it is thus suboptimal, which yields Equation~\eqref{eq-triangle-1}. Because $\bar{\ga}$ is the marginal of $\ga$ we get Equation~\eqref{eq-triangle-2}. Equations~\eqref{eq-triangle-3} and~\eqref{eq-triangle-4} are respectively obtained by the triangle and Minkowski inequalities, which hold because $\Dd$ which is a distance. Equation~\eqref{eq-triangle-5} is the marginalization of $\ga$, and Equation~\eqref{eq-triangle-6} is given by the optimality of $(\bar{\al},\bar{\be})$, which ends the proof of the triangle inequality.

\subsubsection{Proof of the inequality between UGW and CGW}
The proof consists in considering an optimal plan $\pi$ for UGW, building a lift $\al$ of this plan into the cone such that $\Ll(\pi)\geq\Hh(\al)$, and prove that $\al$ is admissible for the program CGW, thus suboptimal.

Using Equation~\eqref{eq-leb-dens-1}, we have
\begin{equation}
\begin{aligned}\label{eq-tensor-leb-dens}
\mu\otimes\mu &= (\f\otimes\f) \pi_1\otimes\pi_1 + (\mu\otimes\mu)^\bot, \\
(\mu\otimes\mu)^\bot &= \mu^\bot\otimes(\f\pi_1) + (\f\pi_1)\otimes\mu^\bot + \mu^\bot\otimes\mu^\bot,\\
\nu\otimes\nu &= (\g\otimes\g) \pi_2\otimes\pi_2 + (\nu\otimes\nu)^\bot,\\
(\nu\otimes\nu)^\bot &= \nu^\bot\otimes(\g\pi_2) + (\g\pi_2)\otimes\nu^\bot + \nu^\bot\otimes\nu^\bot.
\end{aligned}
\end{equation}

Recall that the canonic injection $\margc$ reads $\margc(x,r)=[x,r]$. Based on the above Lebesgue decomposition, we define the conic plan
\begin{equation}
\begin{aligned}\label{eq-plan-conic-lifted}
\al &= (\margc(x, \f(x)^{\frac{1}{p}}), \margc(y,\g(y)^{\frac{1}{p}}))_\sharp\pi(x,y) + \delta_{\zc_X}\otimes\margc_\sharp[\nu^\bot\otimes\de_1] + \margc_\sharp[\mu^\bot\otimes\de_1]\otimes\delta_{\zc_Y}.
\end{aligned}
\end{equation}

We have that $\al\in\Uu_p(\mu,\nu)$. Indeed for the first marginal (and similarly for the second) we have for any test function $\xi(x)$
\begin{align*}
	\int\xi(x)(r)^p\d\al &= \int\xi(x)\f(x)\d\pi_1(x) + 0 + \int\xi(x)(1)^p\d\mu^\bot(x)\\
	&=\int\xi(x)\d(\f(x)\pi_1 + \mu^\bot)\\
	&=\int\xi(x)\d\mu(x).
\end{align*}

We define $\theta^*=\theta^*_c(r,s)$ the parameter which verifies $H_c(r,s)=\theta^* L_c(r/\theta^*, s/\theta^*)$. We restrict $\al\otimes\al$ to the set $S=\{\theta^*_{\lambda(\Gamma)}((rr')^p, (ss')^p)>0\}$. By construction, $\theta^*_c(r,s)$ is 1-homogeneous in $(r,s)$. Thus on S we necessarily have $r,r',s,s' >0$. It yields
\begin{align*}
	\rest{\al\otimes\al}{S} &= (\margc(x, \f(x)^{\frac{1}{p}}), \margc(y,\g(y)^{\frac{1}{p}}),\margc(x', \f(x')^{\frac{1}{p}}), \margc(y',\g(y')^{\frac{1}{p}}))_\sharp(\pi\otimes\pi).
\end{align*}

Concerning the orthogonal part of the decomposition, note that whenever $\theta^*=0$, due to the definition of $H$ the cone distance reads 
\begin{align}
	\Dd([x,r], [y,s])^q = \phi(0)(r^p + s^p).
\end{align}
It geometrically means that the shortest path between $[x,r]$ and $[y,s]$ must pass via the apex, which corresponds to a pure mass creation/destruction regime.

Furthermore we have that
\begin{align*}
	|(\mu\otimes\mu)^\bot| &= \int (r\cdot r')^p \d\rest{(\al\otimes\al)}{S^c},\\
	|(\nu\otimes\nu)^\bot| &= \int (s\cdot s')^p \d\rest{(\al\otimes\al)}{S^c}.
\end{align*}
Indeed, thanks to Equation~\eqref{eq-plan-conic-lifted} we have for the first marginal that
\begin{align*}
	|(\mu\otimes\mu)^\bot| &= \big(\mu^\bot\otimes(\f\pi_1) + (\f\pi_1)\otimes\mu^\bot + \mu^\bot\otimes\mu^\bot\big)(X^2)\\
	&= \int (rr')^p\d\margc_\sharp[\mu^\bot\otimes\de_1]\d\margc(x', \f(x')^{\frac{1}{p}})_\sharp\pi_1(x')\\
	&\qquad+ \int (rr')^p\d\margc(x, \f(x)^{\frac{1}{p}})_\sharp\pi_1(x)\d\margc_\sharp[\mu^\bot\otimes\de_1]\\
	&\qquad+ \int (rr')^p\d\margc_\sharp[\mu^\bot\otimes\de_1]\d\margc_\sharp[\mu^\bot\otimes\de_1]\\
	&= \int (rr')^p \d\rest{(\al\otimes\al)}{S^c}.
\end{align*}
Note that the last equality holds because each term of $\al\otimes\al$ involving a measure $\delta_{\zc_{X}}$ cancels out when integrated against $(rr')^p$.

Eventually the computation gives (thanks to Lemma~\ref{lem-rewrite-ugw})
\begin{align*}
	\Ll(\pi) &= \int_{X^2 \times Y^2} L_{\C(\Gamma)}(\f\otimes\f,\g\otimes\g)\d\pi\d\pi+\phi(0)(|(\mu\otimes\mu)^\bot| +|(\nu\otimes\nu)^\bot|)\\
	&\geq \int H_{\C(\Gamma)}(\f\otimes\f , \g\otimes\g)\d\pi\d\pi  + \phi(0)(|(\mu\otimes\mu)^\bot| +|(\nu\otimes\nu)^\bot|)\\
	&\geq \int \Dd([d_X(x, x'), (\f\otimes\f)^{\frac{1}{p}}], [d_Y(y,y'), (\g\otimes\g)^{\frac{1}{p}}])^q\d\pi\d\pi\nonumber\\
	&\qquad+ \int \phi(0)(rr')^p \d\rest{(\al\otimes\al)}{S^c} + \int \phi(0)(ss')^p \d\rest{(\al\otimes\al)}{S^c}\\
	&\geq \int \Dd([d_X(x, x'), rr'], [d_Y(y,y'), ss'])^q\d\rest{(\al\otimes\al)}{S}\nonumber\\
	&\qquad+ \int \phi(0)((rr')^p  + (ss')^p)\d\rest{(\al\otimes\al)}{S^c}\\
	&\geq \int \Dd([d_X(x, x'), rr'], [d_Y(y,y'), ss'])^q\d\al\d\al\\
	&\geq \Hh(\al).
\end{align*}
Thus we have $\UGW(\Xx,\Yy)=\Ll(\pi)\geq\Hh(\al)\geq\CGW(\Xx,\Yy)$.
\newpage
% !TEX root = ../iclr2021_submission.tex

\section{Optimization, algorithms and formulas}
\label{appendix-algo}

We present in this section the important results of Section~\ref{sec-algo}. We start with Theorem~\ref{ThKonno'sGeneralization} stating that for a wide range of quadratic programs, performing a bi-convex relaxation yields the same objective value as the original program. We prove its application in Theorem~\ref{ThTightLowerBound}. We provide a decomposition property of $\KL^\otimes$, followed by the proof of Proposition~\ref{prop-alternate-simple}, and a description of the algorithm in a discrete setting, where computationnaly implementable formulas are provided.

\subsection{Proof of Theorem~\ref{ThKonno'sGeneralization}}
\begin{proof}
The function $\Ff$ is the symmetrization of $\Ll$, so that $\Ff(\pi,\pi) = \Ll(\pi)$. By the hypothesis on $\Ll$, the mimimum values of the functions (if it exists) are finite.
The two following inequalities are obtained by optimality of $(\pi_*,\ga_*)$,
\begin{equation}
\begin{cases}\label{EqTwoIneq}
\Ff(\pi_*,\ga_*) \leq \Ff(\pi_*,\pi_*) \\
\Ff(\pi_*,\ga_*) \leq \Ff(\ga_*,\ga_*)\,. 
\end{cases}
\end{equation}
Note that the hypotheses imply that $\Ff(\pi_*,\pi_*)$ and $\Ff(\ga_*,\ga_*)$ are both finite.
Combining these two inequalities leads to 
$
\Ff(\pi_*,\pi_*) + \Ff(\ga_*,\ga_*) - 2\Ff(\pi_*,\ga_*) \geq 0\,,
$
which implies
\begin{equation}
\frac 12 \langle \pi_* - \ga_*, k (\pi_*-\ga_*) \rangle \geq 0\,,
\end{equation}
since the separable parts in $\Ff$ cancel. Since $k$ is negative, we also have the converse inequality, thus $\frac 12 \langle \pi_* - \ga_*, k (\pi_*-\ga_*) \rangle= 0$.
Therefore, we deduce when $k$ is definite that $\pi_* = \ga_*$.
\par
We now treat the case when $k$ is not definite. In this case, we only have $\frac 12 \langle \pi_* - \ga_*, k (\pi_*-\ga_*) \rangle= 0$ which implies that $\pi_* - \ga_* \in \operatorname{Ker}(k)$ since $k$ is non positive. The first inequality in \eqref{EqTwoIneq} implies $f(\pi_*) \leq f(\ga_*)$ and by symmetry $f(\pi_*) = f(\ga_*)$ and as a conclusion $\Ff(\pi_*,\pi_*) = \Ff(\pi_*,\ga_*) = \Ff(\ga_*,\ga_*)$.
\par
The last case follows from the observation that
on the segment $[\pi_*,\ga_*] \subset C$, the quadratic part of $\Ff$ is constant. Indeed one has for $t\in [0,1]$, for $z=t(\pi_* - \ga_*) + \ga_*$ one has
\begin{align*}
	\dotp{z}{k(z)} = t^2\dotp{(\pi_* - \ga_*)}{k(\pi_* - \ga_*)} + 2t\dotp{\ga_*}{k(\pi_* - \ga_*)} + \dotp{\ga_*}{k(\ga_*)}
	=\dotp{\ga_*}{k(\ga_*)},
\end{align*}
since $\pi_* - \ga_* \in \operatorname{Ker}(k)$.
 Thus minimizing $\Ff$ on $[\pi_*,\ga_*]$ is reduced to the minimization of $f$ on this segment. By the above remark, $f(\pi_*) = f(\ga_*)$ which implies $\pi_* = \ga_*$ by strict convexity.
\end{proof}

\subsection{Properties of the quadratic KL divergence}

We present in this section an additional property on the quadratic-KL divergence which allows to reduce the computational burden to evaluate it by involving the computation of a standard $\KL$ divergence. 

\begin{proposition}\label{prop-decompose-kl}
	For any measures $(\mu,\nu)\in\Mm_+(\Xx)$, one has
	\begin{equation}
	\begin{aligned}
	\KL(\mu\otimes\nu|\al\otimes\be) &= m(\nu)\KL(\mu|\al) +  m(\mu)\KL(\nu|\be)\\
	&\qquad+ (m(\mu) - m(\al))(m(\nu) - m(\be)).
	\end{aligned}
	\end{equation}
	In particular,
	\begin{align}
		\KL(\mu\otimes\mu|\nu\otimes\nu) = 2m(\mu)\KL(\mu|\nu) + (m(\mu) - m(\nu))^2.
	\end{align}
\end{proposition}
\begin{proof}
	Assuming $\KL(\mu\otimes\nu|\al\otimes\be)$ to be finite, one has $\mu = \f \al$ and $\nu = \g\be$. It reads
	\begin{align*}
		\KL(\mu\otimes\nu|\al\otimes\be) &= \int \log(\f\otimes\g) \d\mu\d\nu - m(\mu)m(\nu) + m(\al)m(\be)\\
		&= m(\nu)\int\log(\f)\d\mu  + m(\mu)\int\log(\g)\d\nu\nonumber\\
		 &\qquad- m(\mu)m(\nu) + m(\al)m(\be)\\
		&= m(\nu)\big[ \KL(\mu|\al) + m(\mu) - m(\al) \big]\nonumber\\ 
		&\qquad+ m(\mu)\big[ \KL(\nu|\be)+ m(\nu) - m(\be) \big]\nonumber\\ 
		&\qquad- m(\mu)m(\nu) + m(\al)m(\be)\\
		&=m(\nu)\KL(\mu|\al) + m(\mu)\KL(\nu|\be)\\ &\qquad+ m(\mu)m(\nu) - m(\nu)m(\al) - m(\mu)m(\be) + m(\al)m(\be)\nonumber\\
		&=m(\nu)\KL(\mu|\al) +  m(\mu)\KL(\nu|\be)\nonumber\\ 
		&\qquad+ (m(\mu) - m(\al))(m(\nu) - m(\be)).
	\end{align*}
\end{proof}

In the Balanced setting, with $(\mu,\nu)$ probabilities, the regularization reads $\KL^\otimes(\pi|\mu\otimes\nu) = 2\KL(\pi|\mu\otimes\nu)$. Thus (up to a factor 2) we retrieve as a particular case the setting of~\cite{peyre2016gromov}.

\subsection{Proof of Proposition~\ref{prop-alternate-simple}}

We now prove Proposition~\ref{prop-alternate-simple} which applies the above result.
\begin{proposition}\label{prop-app-iterate-cost}
	For a fixed $\ga$, the optimal
	$\pi\in\arg\umin{\pi} \Ff(\pi,\ga) +\epsilon\KL(\pi\otimes\gamma|(\mu\otimes\nu)^{\otimes 2})$
	is the solution of
	%	\eq{
	$\umin{\pi} \int c^\epsilon_\ga(x,y) \d\pi(x,y) + \rho m(\ga) \KL(\pi_1|\mu)
	+ \rho m(\ga) \KL(\pi_2|\nu) + \epsilon m(\ga) \KL(\pi|\mu\otimes\nu)$,
	%	}
	where $m(\ga) \eqdef \ga(X \times Y)$ is the total mass of $\ga$, and
	where we define the cost and weight associated to $\ga$ as
	\begin{align*}
		c^\epsilon_\ga(x,y) \eqdef \int \C(\Gamma(x,\cdot,y,\cdot))\d\ga &+ \rho\int \log(\frac{\d\ga_1}{\d\mu})\d\ga_1 + \rho\int \log(\frac{\d\ga_2}{\d\nu})\d\ga_2\nonumber
		+ \epsilon\int \log(\frac{\d\ga}{\d\mu\d\nu})\d\ga.\nonumber
	\end{align*}
\end{proposition}
\begin{proof}
	First note that $\Ff(\ga,\pi)=\Ff(\pi,\ga)$ so that minimizing with the first or the second argument gives the same solution.
	Setting $\ga$ to be fixed, the rest follows from the factorisation
\begin{align*}
	\KL(\pi_1\otimes\ga_1|\mu\otimes\mu)  &= m(\ga)\KL(\pi_1|\mu) + m(\pi)\KL(\ga_1|\mu) + (m(\ga) - m(\mu))(m(\pi) - m(\mu))\\
	& = m(\pi)\Big[ \KL(\ga_1|\mu) + m(\ga) - m(\mu) \Big] + m(\ga)\KL(\pi_1|\mu) - m(\ga)m(\mu)\\
	& = m(\pi) \int \log(\frac{\d\ga_1}{\d\mu})\d\ga_1 + m(\ga)\KL(\pi_1|\mu) - m(\ga)m(\mu)\\
	& = \int \Bigg(\int \log(\frac{\d\ga_1}{\d\mu})\d\ga_1 \Bigg)\d\pi  + m(\ga)\KL(\pi_1|\mu) - m(\ga)m(\mu),
\end{align*}
	and also from
	$\KL(\pi_1|\mu)  = \int\log(\frac{\d\ga_1}{\d\mu})\d\ga_1 - (m(\ga) - m(\mu))$.
	Similar formulas hold for $(\pi_2,\gamma_2)$ and $(\pi,\gamma)$. Summing all $\KL$ terms yields the expression for $c^\epsilon_\ga$.
\end{proof}

\subsection{Discrete setting and formulas}

In order to implement those algorithms, one consider discrete mm-spaces $X=(x_i)_{i=1}^n$ and $Y=(y_j)_{j=1}^m$, endowed with discrete measures
$\mu=\sum_i \mu_i \de_{x_i}$ and $\nu=\sum_j \nu_j \de_{y_j}$, where $\mu_i,\nu_j \geq 0$. 
The distance matrices are $D^X_{i,i'} \eqdef d_X(x_i,x_{i'})$ and $D^X_{j,j'} \eqdef d_X(y_j,y_{j'})$.
Transport plans are thus also discrete $\pi=\sum_{i,j} \pi_{i,j}\de_{(x_i,y_j)}$.

% To implement those algorithms, we represent mm-spaces in a discrete setting and write $\Xx=(X, d, \mu)$ and $\Yy=(Y,D, \nu)$.  

% The sets $X$ and $Y$ become sets of $N$ and $M$ elements respectively written $(x_i)_i$ and $(y_j)_j$ (e.g. nodes of a graph or points in $\RR^d$). Distances $d$ and $D$ are represented by cost matrices $(d_{ij})_{ij}$ and $(D_{ij})_{ij}$ of dimensions $N^2$ and $M^2$. 

% The measures $\mu$ and $\nu$ are encoded by nonnegative vectors of sizes $N$ and $M$. The dependancy of cost matrices w.r.t. the points of the sets and the discrete measure respectively read
%\begin{align*}
%	C^{(p)}_{ij} = C^{(p)}(x_i^{(p)}, x_j^{(p)})
%	\quad \textrm{and} \quad 
%	\mu^{(p)} = \sum_j \mu^{(p)}_j\de_{x_j^{(p)}}.
%\end{align*}
% It is possible to only implement UGW by storing $(d_{ij}),(\mu_i)$ for $\Xx$ and $(D_{ij}),(\nu_i)$ for $\Yy$. In that case $d_{ij}$ implicitly represents the distance between $x_i$ and $x_j$, and the weight $\mu_i$ represents the amount of mass located at $x_i$.

% In this setting any transport plan $\pi = (\pi_{ij})$ is a $NM$-dimensional matrix where $\pi_{ij}$ represents the mass transported from $x_i$ to $y_j$. 

The functional $\Ll$ now reads in this discrete setting
\begin{align*}
	\int(d_X(x,x') - d_Y(y,y'))^2\d\pi(x,y)\d\pi(x',y') =  \sum_{i,j,k,\ell} (D_{i,j}^X - D_{k,\ell}^Y)^2\pi_{i,k}\pi_{j,\ell},
\end{align*}
\begin{align*}
	\qandq \KL(\pi_1\otimes\pi_1 | \mu\otimes\mu) &=  \sum_{i,j}\log\Big(\frac{\pi_{1,i}\pi_{1,j}}{\mu_i\mu_j}\Big) \pi_{1,i}\pi_{1,j} - \sum_{i,j} \pi_{1,i}\pi_{1,j} + \sum_{i,j} \mu_i\mu_j\\
	&=  2m(\pi)\sum_{i}\log\Big(\frac{\pi_{1,i}}{\mu_i}\Big) \pi_{1,i} -m(\pi)^2 + m(\mu)^2,
\end{align*}
where we define the marginals $\pi_{1,k} \eqdef \sum_j \pi_{k,j}$, $\pi_{2,\ell} \eqdef \sum_i \pi_{i,\ell}$ and $m(\pi)=\sum_{i,j} \pi_{i,j}$.

When one runs the stabilized implementation of Sinkhorn's iterations with a ground cost $C_{i,j}=C(x_i,y_j)$ between the points, it is necessary to use a Log-Sum-Exp reduction which reads 
\begin{align}\label{eq-stable-lse}
	\f_i \leftarrow -\frac{\epsilon\rho}{\epsilon + \rho} \text{LSE}_j\big[(g_j - C_{i,j}) / \epsilon + \log(\mu_{j})\big]
\end{align}
where $\text{LSE}_j$ is a reduction performed on the index $j$. It reads
\begin{align}
	\text{LSE}_j(C_{i,j}) \eqdef \log \Big(\sum_j \exp(C_{i,j} - \max_k C_{i,k})\Big) + \max_k C_{i,k},
\end{align}
where the logarithm and exponential are pointwise operations.

\begin{algorithm}[tb]
	\caption{-- \textbf{UGW($\Xx$, $\Yy$, $\rho$, $\epsilon$)} in discrete form}
	\textbf{Input:} mm-spaces $\Xx = (D^X_{i,j}, (\mu_i)_i)$ and $\Yy=(D^Y_{i,j}, (\nu_j)_j)$,  relaxation $\rho$, regularization $\epsilon$ \\
	\textbf{Output:} approximation $(\pi,\ga)$ minimizing~\ref{eq-lower-bound}
	\begin{algorithmic}[1]
		\STATE Initialize matrix $\pi_{i,j}=\ga_{i,j}=\mu_i\nu_j / \sqrt{(\sum_i \mu_i) (\sum_j \nu_j)}$, vector $g^{(s=0)}_j=0$.
		\WHILE{$\pi$ has not converged}
		\STATE Update $\pi\leftarrow\ga$
		\STATE Define $m(\pi)\leftarrow \sum_{i,j} \pi_{i,j}$, $\tilde{\rho} \leftarrow m(\pi)\rho$, $\tilde{\epsilon} \leftarrow m(\pi)\epsilon$
		\STATE Define $c \leftarrow$ ComputeCost($\Xx$, $\Yy$, $\pi$, $\rho$, $\epsilon$)
		\vspace*{0.05cm}
		\WHILE{$(\f,\g)$ has not converged}
		\STATE  $\f\leftarrow -\frac{\tilde{\epsilon}\tilde{\rho}}{\tilde{\epsilon} + \tilde{\rho}}
		\log\Big[\sum_j \exp\big((\g_j - c_{i,j}) / \tilde{\epsilon} + \log\nu_j \big)\Big]$
		\STATE  $\g\leftarrow -\frac{\tilde{\epsilon}\tilde{\rho}}{\tilde{\epsilon} + \tilde{\rho}}
		\log\Big[\sum_i \exp\big((\f_i - c_{i,j}) / \tilde{\epsilon} + \log\mu_i \big)\Big]$
		\ENDWHILE
		\STATE Update $\ga_{i,j}
		\leftarrow \exp\Big[ (\f_i+\g_j-c_{i,j}) / \tilde{\epsilon} \Big]\mu_i\nu_j$
		\STATE Rescale $\ga\leftarrow \sqrt{m(\pi) / m(\ga)} \ga$
		\ENDWHILE
		\STATE Return $(\pi,\ga)$.
	\end{algorithmic}
\end{algorithm}

We also provide an algorithm that computes the cost $c^\epsilon_{\pi}$ defined in Proposition~\eqref{prop-app-iterate-cost}. We focus on the case $\D_\phi=\rho\KL$ and $\C(t)=t^2$ which is computable with complexity $O(n^3)$ as shown in~\citet{peyre2016gromov}. Indeed, note that one has
\begin{align*}
	\int (d_X(x,x') - d_Y(y,y'))^2\d\pi(x',y') &= \int d_X(x,x')^2\d\pi_1(x') + \int d_Y(y,y')^2\d\pi_2(y')\\ 
	&- 2 \int d_X(x,x')d_Y(y,y')\d\pi(x',y').
\end{align*}

\begin{algorithm}[H]
	\caption{-- \textbf{ComputeCost($\Xx$, $\Yy$, $\pi$, $\rho$, $\epsilon$)} in discrete form}
	\textbf{Input:} mm-spaces $\Xx = (D^X_{i,j}, (\mu_i)_i)$ and $\Yy=(D^Y_{k,\ell}, (\nu_j)_j)$, transport matrix $(\pi_{j,k})_{j,k}$, relaxation $\rho$, regularization $\epsilon$ \\
	\textbf{Output:} cost $c^\epsilon_{\pi}$ defined in Proposition~\ref{prop-app-iterate-cost}
	\begin{algorithmic}[1]
		\STATE Compute $\pi_{1,j} \leftarrow \sum_k \pi_{j,k}$ and $\pi_{2,k} \leftarrow \sum_j \pi_{j,k}$ \COMMENT{$\pi_1=\pi\bm{1}$ and $\pi_2=\pi^\top\bm{1}$}
		\STATE Compute $A_i \leftarrow \sum_{j} (D^X_{i,j})^2 \pi_{1,j}$ \COMMENT{$A = (D^X)^{\circ 2}\pi_1$}
		\STATE Compute $B_\ell \leftarrow \sum_{k} (D^Y_{k,\ell})^2 \pi_{2,k}$  \COMMENT{$B = (D^Y)^{\circ 2}\pi_2$}
		\STATE Compute $C_{i,\ell} \leftarrow \sum_j D^X_{i,j} \big( \sum_k D^Y_{k,\ell}\pi_{j,k} \big)$ \COMMENT{$C = D^X \pi D^Y$}
		\STATE Compute $E \leftarrow \rho \sum_j \log\big(\frac{\pi_{1,j}}{\mu_j}\big)\pi_{1,j}   +   \rho \sum_k \log\big(\frac{\pi_{2,k}}{\nu_k}\big)\pi_{2,k}   +   \epsilon\sum_{j,k} \log\big(\frac{\pi_{jk}}{\mu_j\nu_k}\big)\pi_{j,k}$
		\STATE Return $c^\epsilon_{\pi, i,\ell} \leftarrow A_i + B_\ell - 2 C_{i,\ell} + E$ 
	\end{algorithmic}
\end{algorithm}

\newpage
\section{Supplementary experiments}
\label{sec-app-xp}

We provide in this section details on Section~\ref{sec-xp}. 
We start with supplementary synthetic experiments illustrating various features of $\UGW$.
We present our approach to approximate the distance $\CGW$ using a bi-convex relaxation and alternate minimization. We prove the tightness of this relaxation and provide details on the experiments of Section~\ref{sec-algo}.
Then we provide details on the PU learning experiments.

\subsection{Synthetic experiments}

\paragraph{Robustness to outlier}

Figure~\ref{fig-match} shows another experiment on a 2-D dataset, using the same display convention as in Figure~\ref{fig-weight}. It corresponds to the two moons dataset with additional outliers (displayed in  cyan).
Decreasing the value of $\rho$ (thus allowing for more mass creation/destruction in place of transportation) is able to reduce and even remove the influence of the outliers, as expected. 
Furthermore, using small values of $\rho$ tends to favor ``local structures'', which is a behavior quite different from UW~\eqref{eq-uw}.
Indeed, for UW, $\rho \rightarrow 0$ sets to zero all the mass of $\pi$ outside of the diagonal (points are not transported), while for $\UGW$, it is rather pairs of points with dissimilar pairwise distances which cannot be transported together.

%%%%HERE
\begin{figure}[h!]
	\centering
	\begin{tabular}{c@{}c@{}c@{}c}
		{\includegraphics[width=.24\linewidth]{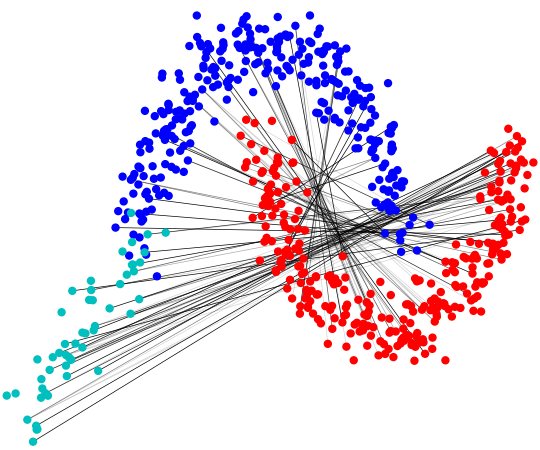}} & 
		{\includegraphics[width=.24\linewidth]{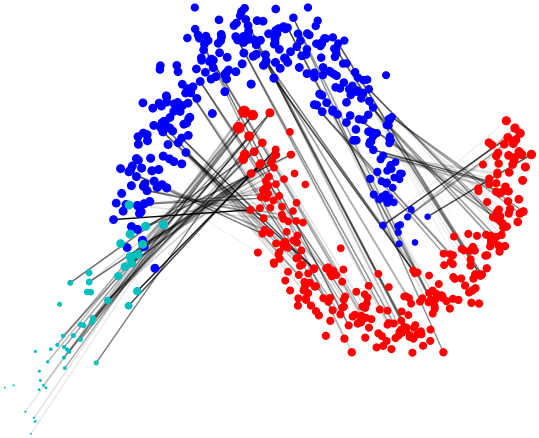}} &
		{\includegraphics[width=.24\linewidth]{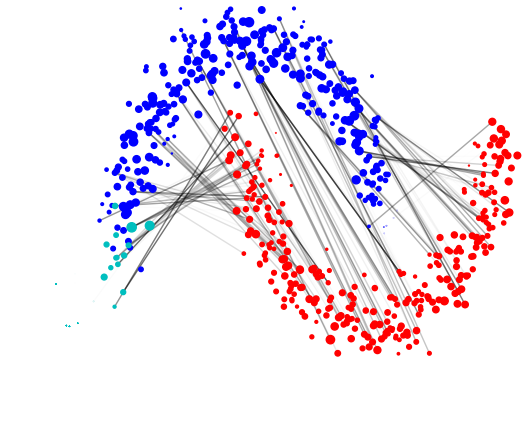}} & 
		{\includegraphics[width=.24\linewidth]{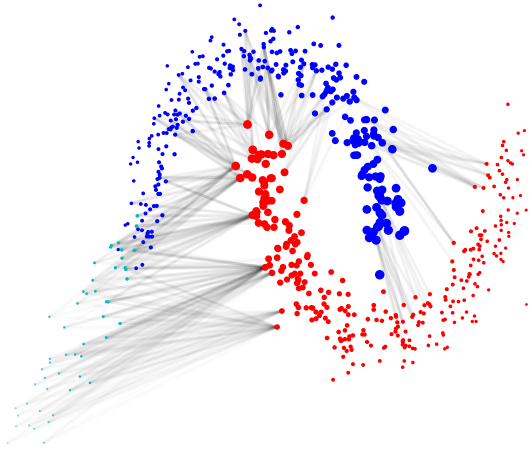}}\\[0mm]
		$\begin{matrix} \GW \\ \rho=\infty \end{matrix}$
		& $\begin{matrix} \UGW \\ \rho=10^0 \end{matrix}$ 
		& $\begin{matrix} \UGW \\ \rho=10^{-1} \end{matrix}$  
		& $\begin{matrix} \text{UW} \\ \rho=10^{-2} \end{matrix}$ 
	\end{tabular}
	\caption{GW and UGW applied to two moons with outliers. A matching using UW is provided to display how invariance to isometries is encoded in the matching.}
	\label{fig-match}
\end{figure}

\paragraph{Graph matching and comparison with Partial-GW.}

We now consider two graphs $(X,Y)$ equipped with their respective geodesic distances. These graphs correspond to points embedded in $\RR^2$, and the length of the edges corresponds to their Euclidean length. These two synthetic graphs are close to be isometric, but differ by addition or modification of small sub-structures. 
The colors $c(x)$ are defined on the ``source'' graph $X$ and are mapped by an optimal plan $\pi$ on $y \in Y$ to a color $\frac{1}{\pi_1(y)} \int_X c(x) \d \pi(x,y)$. This allows to visualize the matching induced by $\GW$ and UGW for a varying $\rho$, as displayed in Figure~\ref{fig-graph}. The graphs for GW should be taken as reference since there is no mass creation. The POT library~\citep{flamary2017pot} is used to compute GW.

For large values of $\rho$, UGW behaves similarly to GW, thus producing irregular matchings which do not preserve the overall geometry of the shapes. 
In sharp contrast, for smaller values of $\rho$ (e.g. $\rho=10^{-1}$), some fine scale structures (such as the target's small circle) are discarded, and UGW is able to produce a meaningful partial matching of the graphs. 
For intermediate values ($\rho=10^0$), we observe that the two branches and the blue cluster of the source are correctly matched to the target, while for GW the blue points are scattered because of the marginal constraint.

%%%%HERE
\newcommand{\myfig}[1]{\includegraphics[width=.22\linewidth]{sections/figures/plots_graph/pic_graph_#1}}
\begin{figure}[!h]
	\centering
	\begin{tabular}{@{}c@{}c@{}c@{}c@{}c@{}}
		\rotatebox{90}{\quad Source X} & \myfig{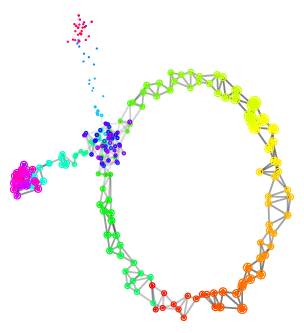} & \myfig{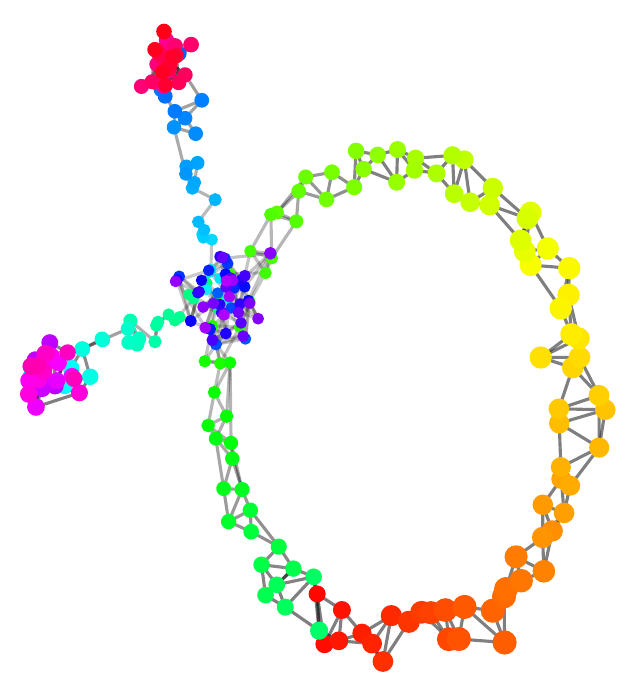} & \myfig{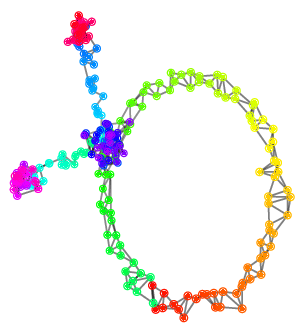} & \myfig{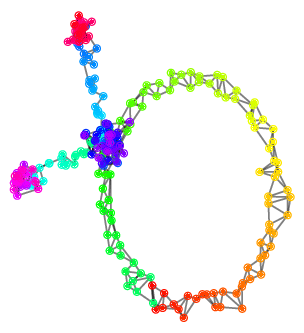} \\
		\rotatebox{90}{\quad Target Y} & \myfig{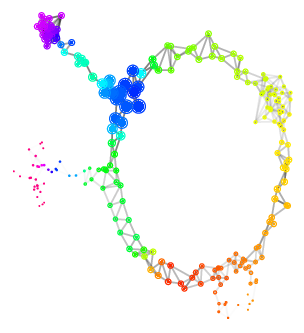} & \myfig{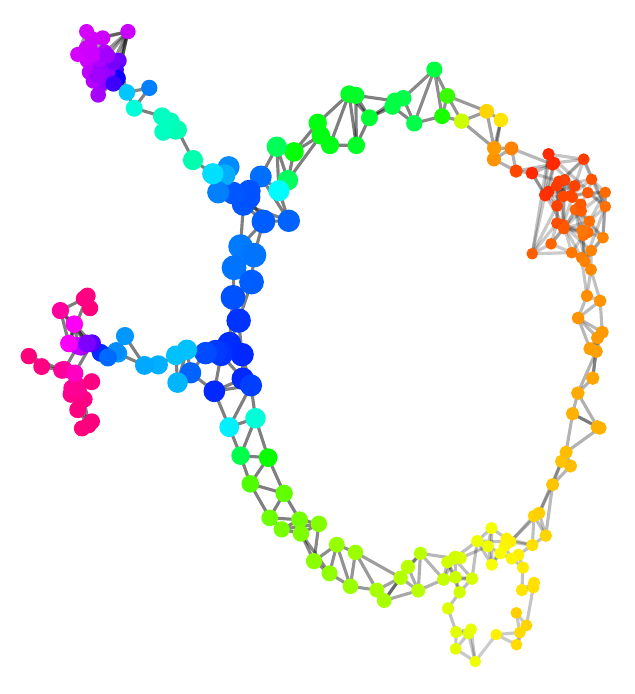} & \myfig{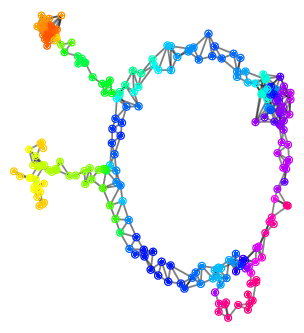} & \myfig{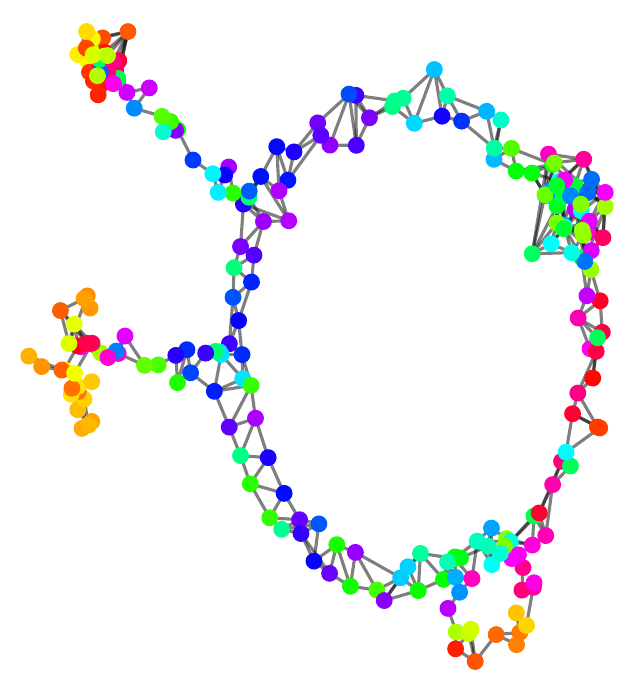} \\
		& $\rho=0.1$ & $\rho=1$ & $\rho=10$ & GW ($\rho=\infty$)
	\end{tabular}
	\caption{Comparison of UGW and GW for graph matching.}
	\label{fig-graph}
\end{figure}

\newcommand{\myfigd}[1]{\includegraphics[width=.2\linewidth]{sections/figures/plots_graph/pic_graph_#1}}
\begin{figure}
	\centering
	\begin{tabular}{@{}c@{}c@{}c@{}c@{}c@{}}
		\rotatebox{90}{\quad Source X} & \myfigd{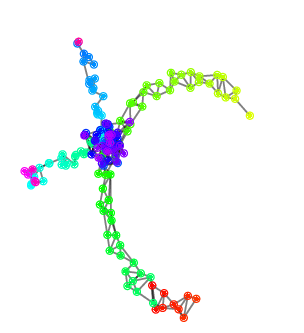} & \myfigd{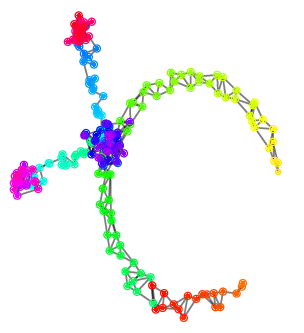} & \myfigd{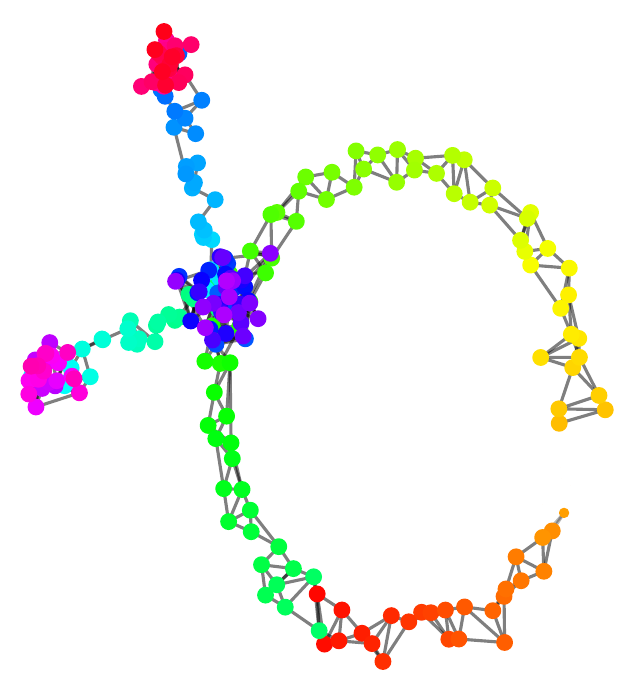} & \myfig{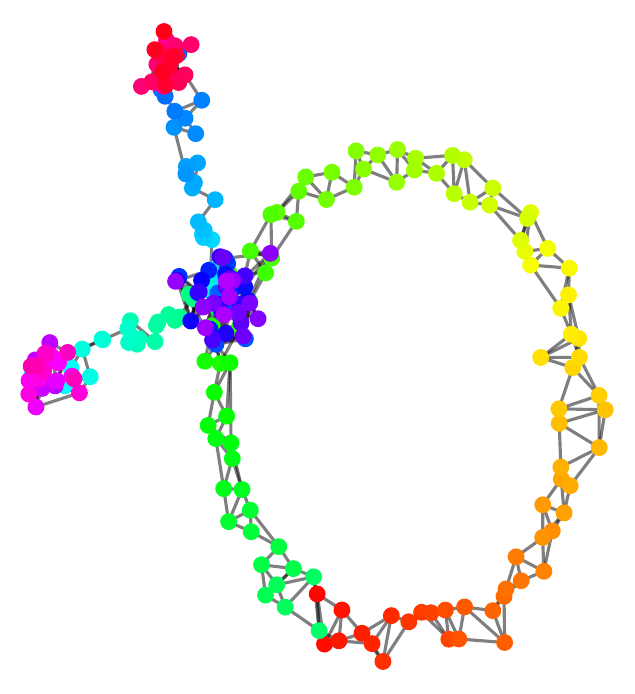} \\
		\rotatebox{90}{\quad Target Y} & \myfigd{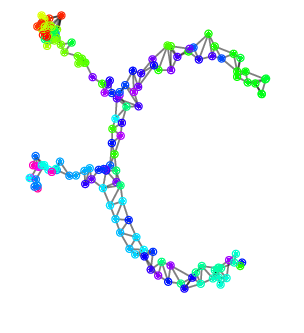} & \myfigd{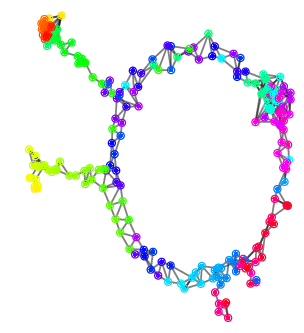} & \myfigd{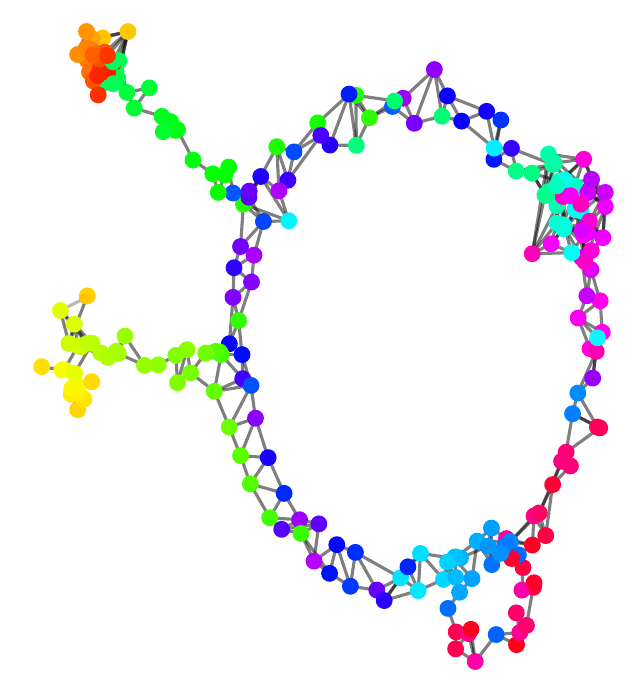} & \myfig{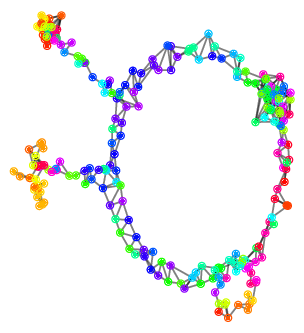} \\
		& m=$0.64$ & m=$0.93$ & m=$0.98$ & GW (m=$1$)
	\end{tabular}
	\caption{Comparison of Partial-GW for graph matching. Here $m$ is the budget of transported mass.}
	\label{fig-graph-2}
\end{figure}

Figure~\ref{fig-graph-2} shows a comparison with Partial-GW~\citep{chapel2020partial}, computed using the POT library. It is close to UGW with a $\TV^\otimes$ penalty, since partial OT is equivalent to the use of a TV relaxation of the marginal. 
UGW with a $\KL^\otimes$ penalty is first computed for a given $\rho$, then the total mass $m$ of the optimal plan is computed, and is used as a parameter for PGW which imposes this total mass as a constraint. Figure~\ref{fig-graph} and~\ref{fig-graph-2} display the transportation strategy associated to both methods. 
KL-UGW operates smooth transitions between transportation and creation of mass, while PGW either performs pure transportation or pure destruction/creation of mass. In Figure~\ref{fig-graph-2} nodes of the graphs are removed and thus ignored by the matching. Note also that since PGW is equivalent to solving GW on sub-graphs, the color distribution of GW and PGW are similar.

\paragraph{Influence of $\epsilon$.}

Figures~\ref{fig-weight}, \ref{fig-match}, \ref{fig-graph} and \ref{fig-graph-2} do not show the influence of $\epsilon$. 
This parameter is set of a low value $\epsilon=10^{-2}$ on a domain $[0,1]^2$ so as to approximate the optimal plan of the unregularized $\UGW$ problem. 
We present now an experiment on graphs which highlights the impact of $(\epsilon,\rho)$ on the plan $\pi$.

We compare two graphs $(\Xx,\Yy)$ displayed Figure~\ref{fig:foobar}. The graph $\Xx$ is composed of two communities of equal size connected with random edges. The graph $\Yy$ is similar to $\Xx$, but the communities are imbalanced and it contains outliers. Moving inside a community costs $1$, reaching another community costs $4$ and reaching an outlier $2$. We equip the mm-space with uniform weights and shortest path distance.

%%%%HERE
\begin{figure}
	\centering
	\begin{tabular}{c@{\hspace{3mm}}c}
		\includegraphics[width=0.3\textwidth]{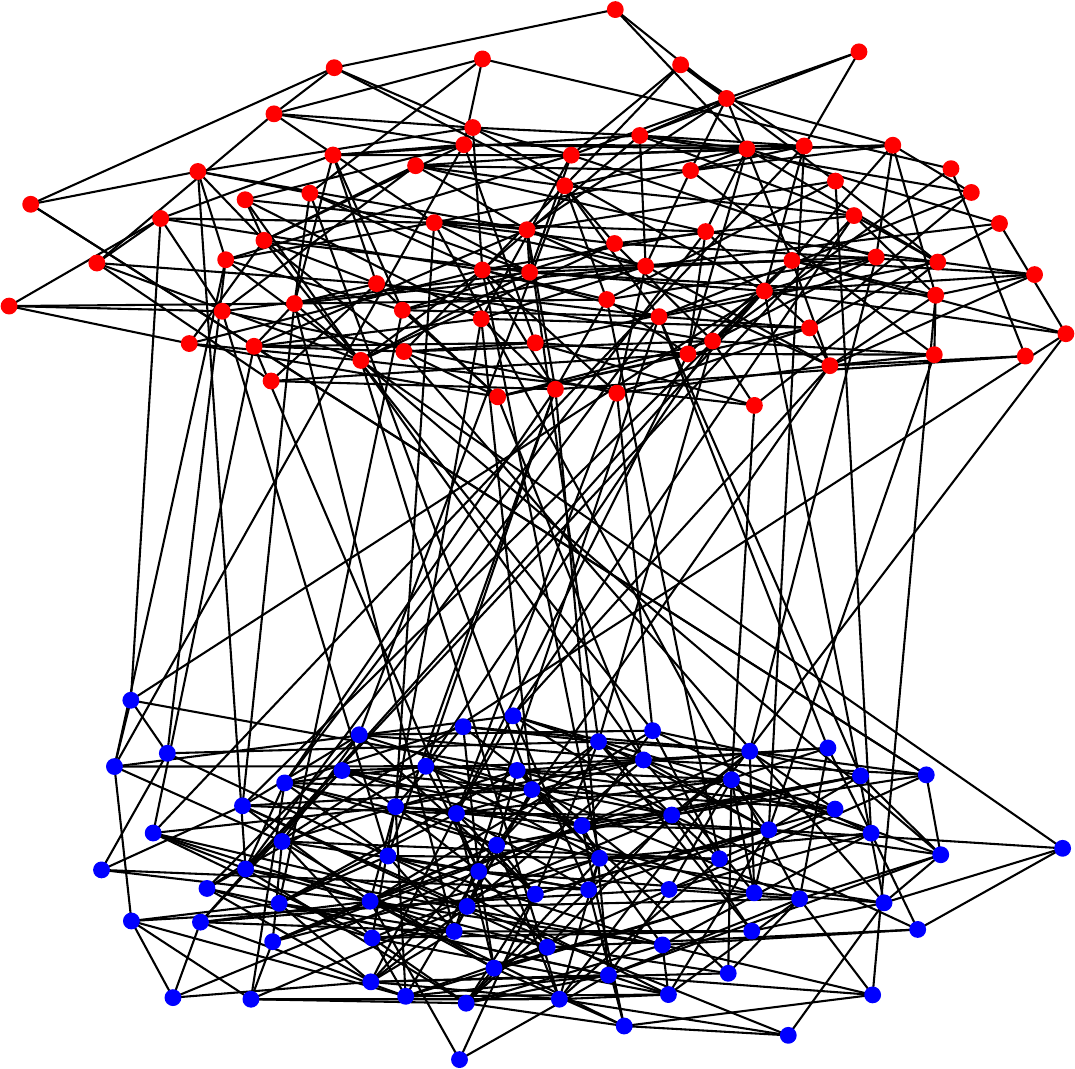}& 
		\includegraphics[width=0.3\textwidth]{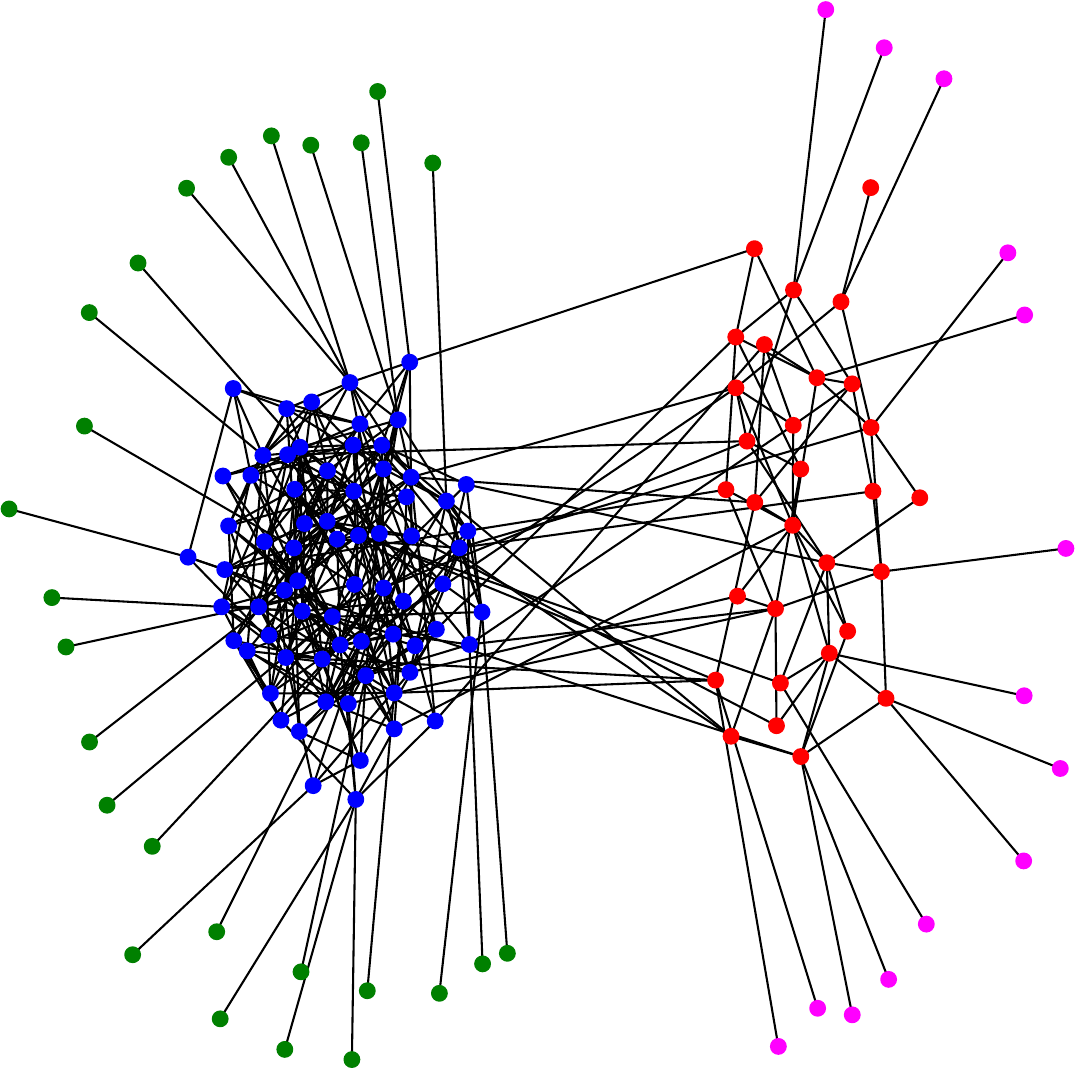}
	\end{tabular}
	%	\subfigure[]{\includegraphics[width=0.2\textwidth]{sections/figures/figures_graph_matching/plot_graph_1.pdf}} 
	%	\subfigure[]{\includegraphics[width=0.2\textwidth]{sections/figures/figures_graph_matching/plot_graph_2.pdf}}
	\caption{Graphs $\Xx$ (left) and $\Yy$ (right) plotted using networkx.}
	\label{fig:foobar}
\end{figure}

We plot in Figure~\ref{fig-graph-match} optimal transport plans $\pi$ for given values of $(\epsilon,\rho)$, including the balanced case $\GW_\epsilon$ where $\rho=\infty$. 
The transport matrix has a block structure: the 2 horizontal blocks correspond to $\Xx$ and its two communities, the 4 vertical blocks corresponds to $\Yy$ (with, from left to right, the large blue community, the small red one, then the pink and green outliers).
Decreasing $\rho$ results in a more structured transport matrix: outliers are removed and inter-community matching is avoided. Again, the marginal constraint of $\GW_\epsilon$ makes the plan more sensitive to structural noise (e.g. outliers) in graphs.
Concerning the parameter $\epsilon$, increasing it creates correlations between pairs of points whose distortion is of order of $\sqrt{\epsilon}$. Indeed, we see for $\epsilon=3$ that correlation between communities and their outliers appear, even for small $\rho$.
Furthermore, when $\epsilon$ is too large the transport becomes uninformative, which highlights a crucial trade-off between computational speed and expressiveness of the transport plan.

%%%%HERE
\newcommand{\myfigu}[2]{\includegraphics[width=.2\linewidth]{sections/figures/figures_graph_matching_cor/plot_matrix_eps#1_rho#2.pdf}}
\begin{figure}
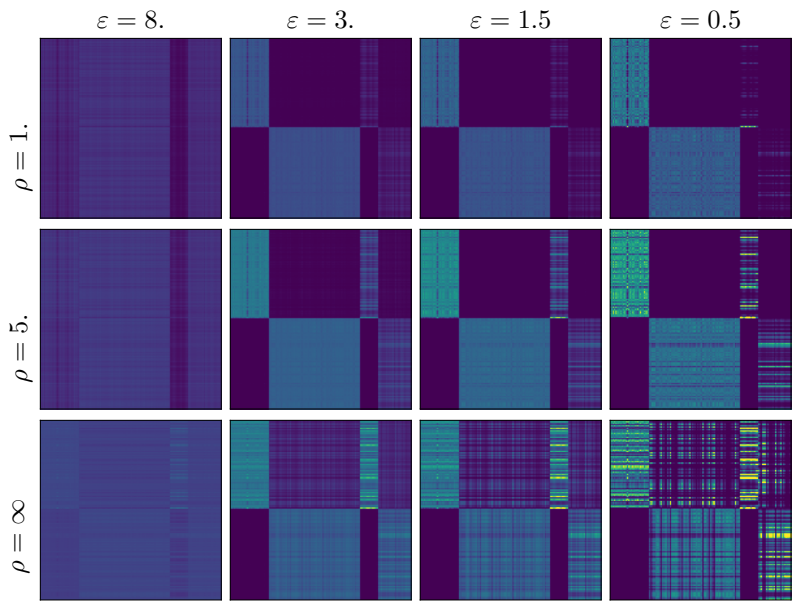

	\centering
	\begin{tabular}{c@{\hspace{1mm}}c@{\hspace{1mm}}c@{\hspace{1mm}}c@{\hspace{1mm}}c}
		& $\epsilon=8.$ & $\epsilon=3.$ & $\epsilon=1.5$ & $\epsilon=0.5$\\[0mm]
		\rotatebox{90}{\quad$\rho=1.$} & {\myfigu{80}{10}} & {\myfigu{30}{10}} & {\myfigu{15}{10}} & {\myfigu{05}{10}}\\[0mm]
		\rotatebox{90}{\quad$\rho=5.$} & {\myfigu{80}{50}} & {\myfigu{30}{50}} & {\myfigu{15}{50}} & {\myfigu{05}{50}}\\[0mm]
		\rotatebox{90}{\quad$\rho=\infty$} & {\myfigu{80}{None}} & {\myfigu{30}{None}} & {\myfigu{15}{None}} & {\myfigu{05}{None}}\\[0mm]
	\end{tabular}
	\caption{Display of the optimal transport plan $\pi$. The color scale is common to all plots.}
	\label{fig-graph-match}
\end{figure}

\subsection{Computation of the CGW distance}
In this section we focus on computing the distance $\CGW$~\eqref{eq-ugw-conic}, which is a quadratic minimization program with linear constraint.
Similar to what is performed with UGW~\ref{sec-algo}, we consider a relaxation using a tensorized conic plan $\al\otimes\be$ with $\al,\be\in\Uu_{p}(\mu,\nu)$.
The minimized cost thus reads
\begin{equation}\label{eq-cgw-relax}
\begin{aligned}
\Hh(\al,\be)\eqdef\int &\Dd([d_X(x,x'), r r'], [d_Y(y,y'), s s'])^q \,\d\al([x,r], [y,s])\d\be([x',r'], [y',s']).
\end{aligned}
\end{equation}
Note that for fixed $\be\in\Uu_{p}(\mu,\nu)$, the minimization w.r.t. $\al$ is a convex linear program with the linear conic constraint set $\Uu_{p}(\mu,\nu)$ and with cost
\begin{equation}\label{eq-conic-local-cost}
\begin{aligned}
\Cc_\Co(x,r,y,s)\eqdef\int &\Dd([d_X(x,x'), r r'], [d_Y(y,y'), s s'])^q \,\d\be([x',r'], [y',s']).
\end{aligned}
\end{equation}
Since we focus on the numerical implementation of CGW, we consider the setting of Gaussian-Hellinger distance which computes the distortion with $\C(t)=t^2$, due to a reduced memory and computation complexity to calculate $|d_X - d_Y|^2$ (see Section~\ref{sec-algo}).
In that case the cone distance reads for a given $\rho$
\begin{equation}\label{eq-gh-conic-cost}
\begin{aligned}
\Dd([d_X(x,x'), r r'], [d_Y(y,y'), s s'])^2 = \rho\Big[ (rr')^2 + (ss')^2 - 2rr'ss'\, e^{- |d_X - d_Y|^2 / 2\rho} \Big].
\end{aligned}
\end{equation}
Before focusing on the discretization of this problem to make it computable, we prove that when $|d_X - d_Y|^2$ is a conditionnaly definite kernel then the above cost is a negative kernel on $\Uu_{p}(\mu,\nu)$. 
Thus Theorem~\ref{ThKonno'sGeneralization} holds.
\begin{proposition}
	Assume that the kernel $|d_X - d_Y|^2$ is conditionnaly negative definite. Then the cost~\eqref{eq-gh-conic-cost} is a negative definite kernel on $\Uu_{p}(\mu,\nu)$.
\end{proposition}
\begin{proof}
	Take any plan $\al\in\Uu_{p}(\mu,\nu)$. Integrating against $(rr')^2$ or $(ss')^2$ yields a constant term. Indeed one has for $(rr')^2$
	\begin{align*}
		\int (rr')^2\d\al([x,r], [y,s])\d\al([x',r'], [y',s']) &= \Bigg( \int (r)^2\d\al([x,r], [y,s]) \Bigg)^2\\
		 &= \Bigg( \int \d\mu(x) \Bigg)^2 \\
		 &= m(\mu)^2.
	\end{align*}
	Thus minimizing w.r.t.~\eqref{eq-gh-conic-cost} is equivalent to minimizing w.r.t. $- 2rr'ss'\, e^{- |d_X - d_Y|^2 / 2\rho}$, which is a product of positive definite kernels $(rr')$, $(ss')$ and $e^{- |d_X - d_Y|^2 / 2\rho}$ thanks to Berg's Theorem~\citet{berg1984harmonic} and because we assume the kernel $|d_X - d_Y|^2$ is c.n.d.
	Due to the extra minus sign we get that the kernel is negative definite, which ends the proof.
\end{proof}

An important point is to implement the constraint set $\Uu_{p}(\mu,\nu)$ which integrates against radial coordinates $(r,s)\in\RR_+^2$.
Such integration is impossible in practice, but thanks to~\citet[Theorem 7.20]{liero2015optimal}, we know that the radius can be restricted to $[0,R]$ where $R^2 = m(\mu)^2 + m(\nu)^2$ (up to a dilation of the plan).
Thus we propose to discretize the constraint by sampling regularly the interval as $\{lR / L, l\in\llbracket 0,L\rrbracket\}$.

We consider discrete mm-spaces as in Section~\ref{appendix-algo}, i.e. mm-spaces noted as $\Xx = (D^X_{i,j}, (\mu_i)_i)$ and $\Yy=(D^Y_{i,j}, (\nu_j)_j)$. 
Write a conic plan $\al_{ijkl}=\al([x_i, r_k],[y_j, s_l])$. The conic constraints read for $k\in\llbracket 0,K\rrbracket$ and $l\in\llbracket 0,L\rrbracket$
\begin{align*}
	\sum_{j,k,l}\Big(\frac{kR}{K}\Big)^2\al_{ijkl} = \mu_i \qandq \sum_{i,k,l}\Big(\frac{lR}{L}\Big)^2\al_{ijkl} = \nu_j.
\end{align*}
The cost $\Cc_\Co$~\eqref{eq-conic-local-cost} is computed via the formula
\begin{align*}
	\Cc_{ijkl} \eqdef \sum_{i',j',k',l'} \rho\Bigg[ (\tfrac{kR}{K}\tfrac{k'R}{K})^2 + (\tfrac{lR}{L}\tfrac{l'R}{L})^2 - 2(\tfrac{kR}{K}\tfrac{k'R}{K}\tfrac{lR}{L}\tfrac{l'R}{L})e^{-|D^X_{i,i'} - D^Y_{j,j'}|^2 / 2\rho} \Bigg]\al_{i'j'k'l'}.
\end{align*}
Eventually, the whole program solving one step of the alternate minimization algorithm is given Equation~\eqref{eq-cgw-discrete}.
The approximation of $\CGW$ is performed by alternatively updating $\al$ and $\Cc_\Co$ until the minimization attains a local minima
\begin{equation}\label{eq-cgw-discrete}
\min_{\al_{ijkl}} 
\left\{
\begin{aligned}
\sum_{i,j,k,l} \Cc_{ijkl}\al_{ijkl} \quad\textrm{s.t.}\quad
\sum_{j,k,l}(\tfrac{kR}{K})^2\al_{ijkl} = \mu_i \qandq \sum_{i,k,l}(\tfrac{lR}{L})^2\al_{ijkl} = \nu_j
\end{aligned}
\right\}.
\end{equation}

\paragraph{Details on the experiments of Section~\ref{sec-xp}.}
One can observe in the above procedure that the memory complexity of $\al$ and $\Cc_\Co$ is prohibitively high to use it in practice, due to the discretization of the radial coordinate which make the size of both tensors scaling as $O(NMKL)$ where $N,M$ are the number of samples in the spaces $(\Xx,\Yy)$.
Thus our experiment are performed considering Euclidean mm-spaces composed of samples $N,M\in\{2,3,5\}$, and we take $K=L=10$.
To guarantee as much as possible that we reach the global minima, we consider $10$ random initializations and $10$ random permutation matrices $P$ lifted as conic plan by setting $\al_{\cdot\cdot kl} = P$ for any $(k,l)$.
The latter initialization is assumed to be close to extremal points of the constraint polytope. Since Theorem~\ref{ThKonno'sGeneralization} holds for Euclidean mm-spaces, the optimal plan is also an extremal point of the polytope.
To compare $\CGW$ with $\UGW$, we set a solver with a level of entropy $\epsilon=10^{-3}$.
In Figure~\ref{fig:cgw-ugw-hist} we set $\rho=10^{-1}$.

\subsection{Details on PU learning experiments}

\paragraph{Details on training for PU learning tasks.}
We present the characteristics of the datasets in Table~\ref{table:data-info}. The variance of the accuracy results presented in Table~\ref{table:data-perf} is presented in Table~\ref{table:data-std}. The computations were made on an internal GPU cluster composed of 10 Tesla K80 and 3 Tesla P100. We also detail the parameters of the numerical solver computing UGW which is the core component of our numerical experiments.
\begin{itemize}
	\item The maximum number of iteration to update the plan is set to $3000$.
	\item The tolerance on convergence of $\pi$ in log-scale is set to $10^{-5}$, i.e. the algorithm stops when $\norm{\log\pi^{t+1} - \log\pi^{t} }_\infty < tol$. 
	\item The maximum number of iteration to update the Sinkhorn potentials is set to $3000$.
	\item The tolerance on convergence of $(\f, \g)$ is set to $10^{-6}$, i.e. the algorithm stops when $\norm{\f^{t+1} - \f^t}_\infty < tol$.
\end{itemize}

\begin{table}
	\begin{center}
		\begin{tabular}{ |c|c|c|c|c| } 
			\hline
			Dataset & \# of samples & \# of positives & Dim. & PCA Dim. \\ 
			\hline
%			mushrooms & 8,124 & 3,916 & 112 \\ 
%			shuttle & 58,000 & 45,586 & 9 \\ 
%			pageblocks & 5,473 & 4,913 & 10 \\ 
%			usps & 9,298 & 1,553 & 256 \\ 
%			connect-4 & 67,557 & 44,473 & 126 \\ 
%			spambase & 4,601 & 1,813 & 57 \\ 
%			\hline
			*-caltech & 1,123 & 151 & surf: 800 / decaf: 4096 & surf: 10 / decaf: 40 \\ 
			*-amazon & 958 & 92 & surf: 800 / decaf: 4096 & surf: 10 / decaf: 40 \\ 
			*-webcam & 295 & 29 & surf: 800 / decaf: 4096 & surf: 10 / decaf: 40 \\ 
			*-dslr & 157 & 12 & surf: 800 / decaf: 4096 & surf: 10 / decaf: 40 \\ 
			\hline
		\end{tabular}
	\end{center}
	\caption{Characteristics of datasets}
	\label{table:data-info}
\end{table}

\paragraph{Initialization for cross-domain tasks.}
To initialize UGW  when the features are different we propose to use a UOT solution of a matching between distance histograms which reads
\begin{equation}\label{eq-def-flb}
\FLB(\Xx,\Yy) \eqdef \min \int_{X\times Y} |\bar{\mu}\star d_X - \bar{\nu}\star d_Y|^2 \d\pi + \rho\KL(\pi_1|\mu) + \rho\KL(\pi_2|\nu) + \epsilon\KL(\pi|\mu\otimes\nu),
\end{equation}
where $\mu\star d_X(x) \eqdef \int d_X(x,x')\d\mu(x')$ is the eccentricity, i.e. a histogram of aggregated distances, and $\bar{\mu} = \mu / m(\mu)$. 
In~\citet{memoli2011gromov} this relaxation is refered as FLB and is a lower bound of GW, but in our unbalanced setting this program cannot a priori be compared with UGW.

\paragraph{Reducing the number of parameters.}

In Table~\ref{table:data-perf}, the accuracy for $\UGW$ is performed by selecting a pair of parameters $(\rho_1,\rho_2)$ for each task via a validation protocol detailed Section~\ref{sec-xp}.
It is desirable to reduce the number of parameters, to see if the performance does not significantly decrease, and avoid overparameterization of the task.
We propose in this section two strategies
The first case keeps one pair $(\rho_1,\rho_2)$ over all tasks.
The second case keeps a pair for each pair of domain tasks (i.e. surf$\leftrightarrow$surf, decaf$\leftrightarrow$decaf, surf$\leftrightarrow$decaf and decaf$\leftrightarrow$surf) for a total of 8 parameters, which allows to normalize adaptively each dataset via an adapted choice of parameters $(\rho_1,\rho_2)$.
The validation protocol is modified since we aggregate accuracies from different tasks.
The selected parameters are obtained by taking the highest mean excess accuracy over all tasks, where the excess is defined by comparing the accuracy to the case where we only predict false positives.
This measure of performance is computed on the validation folds, and we report the accuracy over the testing folds in Table~\ref{table:data-appendix}.

\begin{table}[]
	\begin{center}
		\resizebox{\textwidth}{!}{
			\begin{tabular}{|c|c|c|c|c|c|c|}
				\hline
				Dataset                      & prior & Init (PW) & PGW  & \textbf{UGW} & UGW (2 param.) & UGW (8 param.) \\ 
				\hline
				surf-C $\rightarrow$ surf-C  & 0.1   & \textbf{89.9} & 84.9 & 83.9 & 81.8 & 83.9 \\   
				surf-C $\rightarrow$ surf-A  & 0.1   & 81.8 & 82.2 & \textbf{83.5} & 83.1 & 83.3 \\   
				surf-C $\rightarrow$ surf-W  & 0.1   & \textbf{81.9} & 81.3 & 80.3 & 80.1 & 80.4 \\   
				surf-C $\rightarrow$ surf-D  & 0.1   & 80.0 & 81.4 & \textbf{83.2} & 80.3 & \textbf{83.2} \\  
				\hline
				surf-C $\rightarrow$ surf-C  & 0.2   & \textbf{79.7} & 75.7 & 75.4 & 67.5 & 75.4 \\   
				surf-C $\rightarrow$ surf-A  & 0.2   & 65.6 & 66.0 & \textbf{76.4} & 74.0 & 73.0 \\   
				surf-C $\rightarrow$ surf-W  & 0.2   & 65.1 & 64.3 & \textbf{67.3} & 63.8 & 64.9 \\   
				\hline
				decaf-C $\rightarrow$ decaf-C & 0.1   & \textbf{93.9} & 83.0 & 86.8 & 84.8 & 84.8 \\   
				decaf-C $\rightarrow$ decaf-A & 0.1   & 80.1 & 81.4 & \textbf{85.6} & 83.7 & 83.7 \\   
				decaf-C $\rightarrow$ decaf-W & 0.1   & 80.1 & 82.7 & \textbf{86.1} & 85.6 & 85.6 \\   
				decaf-C $\rightarrow$ decaf-D & 0.1   & 80.6 & \textbf{83.8} & 83.4 & 83.6 & 83.6 \\   
				\hline
				decaf-C $\rightarrow$ decaf-C & 0.2   & \textbf{90.6} & 76.7 & 80.5 & 75.7 & 75.7 \\   
				decaf-C $\rightarrow$ decaf-A & 0.2   & 62.5 & 68.7 & 74.7 & \textbf{75.0} & \textbf{75.0} \\   
				decaf-C $\rightarrow$ decaf-W & 0.2   & 65.7 & 75.9 & 79.2 & \textbf{80.2} & \textbf{80.2} \\   
				\hline\hline
				 Dataset                       & prior & Init (FLB) & PGW  & \textbf{UGW} & UGW (2 param.)& UGW (8 param.) \\
				 \hline
				 surf-C $\rightarrow$ decaf-C & 0.1   & 85.0 & 85.1 & \textbf{85.6} & 85.0 & 85.0  \\
				 surf-C $\rightarrow$ decaf-A & 0.1   & 84.2 & \textbf{87.1} & 83.6 & 83.5 & 83.5  \\
				 surf-C $\rightarrow$ decaf-W & 0.1   & 86.2 & \textbf{88.6} & 86.8 & 87.4 & 87.4  \\
				 surf-C $\rightarrow$ decaf-D & 0.1   & 84.7 & \textbf{91.1} & 90.7 & 89.3 & 89.3  \\
				\hline
				 surf-C $\rightarrow$ decaf-C & 0.2   & 74.8 & 75.6 & 75.9 & \textbf{76.2} & \textbf{76.2}   \\
				 surf-C $\rightarrow$ decaf-A & 0.2   & 76.2 & \textbf{87.9} & 82.4 & 83.2 & 83.2   \\
				 surf-C $\rightarrow$ decaf-W & 0.2   & 81.5 & 88.4 & \textbf{89.9} & 88.8 & 88.8   \\
				\hline
				 decaf-C $\rightarrow$ surf-C  & 0.1   & 81.7 & 81.0 & 81.1 & 81.9 & \textbf{82.1}   \\
				 decaf-C $\rightarrow$ surf-A  & 0.1   & 80.9 & 81.2 & \textbf{82.4} & 81.2 & 82.1  \\
				 decaf-C $\rightarrow$ surf-W  & 0.1   & 82.0 & 81.3 & \textbf{83.5} & 80.8 & 80.7   \\
				 decaf-C $\rightarrow$ surf-D  & 0.1   & 80.0 & 80.8 & \textbf{81.5} & 80.0 & \textbf{81.5}   \\
				\hline
				 decaf-C $\rightarrow$ surf-C  & 0.2   & 66.6 & 63.7 & 65.2 & 66.5 & \textbf{67.9}   \\
				 decaf-C $\rightarrow$ surf-A  & 0.2   & 62.9 & 62.4 & \textbf{69.3} & 62.2 & 68.5   \\
				 decaf-C $\rightarrow$ surf-W  & 0.2   & 65.1 & 61.4 & \textbf{83.3} & 61.1 & 65.0  \\
				 \hline
			\end{tabular}
		}
	\end{center}
	\caption{Accuracy for all tasks. The left block are domain adaptation experiments with similar features, where both PGW and UGW are initialised with PW. The right block are domain adaptation experiments with different features, and the reported init is FLB (see Appendix~\ref{sec-app-xp}) used for UGW.}
	\label {table:data-appendix}
\end{table}

\begin{table}[]
	\begin{center}
			\begin{tabular}{|c|c|c|c|c|c|c|}
				\hline
				Dataset                      & prior & Init (PW) & PGW  & \textbf{UGW} & UGW (2 param.) & UGW (8 param.) \\
				\hline
				surf-C $\rightarrow$ surf-C  & 0.1   & 2.05 & 1.95 & 2.93 & 2.14 & 2.94  \\
				surf-C $\rightarrow$ surf-A  & 0.1   & 1.25 & 1.89 & 2.14 & 2.29 & 3.33  \\ 
				surf-C $\rightarrow$ surf-W  & 0.1   & 1.33 & 1.82 & 0.73 & 0.45 & 0.82  \\ 
				surf-C $\rightarrow$ surf-D  & 0.1   & 0.00 & 1.69 & 2.63 & 0.73 & 2.63  \\ 
				\hline
				surf-C $\rightarrow$ surf-C  & 0.2   & 2.98 & 4.66 & 5.07 & 2.42 & 5.07  \\ 
				surf-C $\rightarrow$ surf-A  & 0.2   & 2.87 & 3.29 & 3.59 & 2.15 & 9.46  \\ 
				surf-C $\rightarrow$ surf-W  & 0.2   & 1.95 & 2.12 & 9.22 & 1.82 & 7.61  \\ 
				\hline
				decaf-C $\rightarrow$ decaf-C & 0.1   & 1.61 & 2.24 & 2.46 & 1.64 & 1.64 \\ 
				decaf-C $\rightarrow$ decaf-A & 0.1   & 0.44 & 1.91 & 4.52 & 2.08 & 2.08 \\ 
				decaf-C $\rightarrow$ decaf-W & 0.1   & 0.44 & 2.55 & 1.65 & 1.90 & 1.90  \\ 
				decaf-C $\rightarrow$ decaf-D & 0.1   & 0.92 & 1.54 & 2.06 & 1.67 & 1.67  \\ 
				\hline
				decaf-C $\rightarrow$ decaf-C & 0.2   & 2.54 & 3.59 & 5.73 & 2.53 & 2.54  \\ 
				decaf-C $\rightarrow$ decaf-A & 0.2   & 2.09 & 4.39 & 7.46 & 4.52 & 4.52 \\ 
				decaf-C $\rightarrow$ decaf-W & 0.2   & 1.93 & 3.60 & 5.89 & 3.61 & 3.61 \\ 
				\hline\hline
				Dataset                      & prior & Init (PW) & PGW  & \textbf{UGW} & UGW (2 param.) & UGW (8 param.) \\
				\hline
				surf-C $\rightarrow$ decaf-C & 0.1   & 2.79 & 2.64 & 3.01 & 2.71 & 2.71   \\
				surf-C $\rightarrow$ decaf-A & 0.1   & 2.08 & 6.50 & 3.28 & 2.82 & 2.82   \\
				surf-C $\rightarrow$ decaf-W & 0.1   & 1.89 & 5.63 & 3.97 & 3.62 & 3.62  \\
				surf-C $\rightarrow$ decaf-D & 0.1   & 1.93 & 8.09 & 7.09 & 7.46 & 7.46   \\
				\hline
				surf-C $\rightarrow$ decaf-C & 0.2   & 2.56 & 3.32 & 4.02 & 3.66 & 3.66   \\
				surf-C $\rightarrow$ decaf-A & 0.2   & 3.74 & 6.61 & 10.5 & 8.04 & 8.04   \\
				surf-C $\rightarrow$ decaf-W & 0.2   & 2.75 & 5.82 & 3.33 & 3.64 & 3.64   \\
				\hline
				decaf-C $\rightarrow$ surf-C  & 0.1   & 1.82 & 1.61 & 1.21 & 1.77 & 2.29  \\
				decaf-C $\rightarrow$ surf-A  & 0.1   & 1.18 & 1.94 & 2.11 & 1.36 & 2.10  \\
				decaf-C $\rightarrow$ surf-W  & 0.1   & 1.67 & 2.03 & 3.94 & 1.01 & 1.17   \\
				decaf-C $\rightarrow$ surf-D  & 0.1   & 0.00 & 1.60 & 1.70 & 0.00 & 1.70   \\
				\hline
				decaf-C $\rightarrow$ surf-C  & 0.2   & 3.04 & 2.92 & 7.21 & 3.24 & 4.08   \\
				decaf-C $\rightarrow$ surf-A  & 0.2   & 1.84 & 4.54 & 5.92 & 2.04 & 5.19   \\
				decaf-C $\rightarrow$ surf-W  & 0.2   & 2.86 & 3.23 & 6.43 & 1.52 & 3.76  \\
				\hline
			\end{tabular}
	\end{center}
	\caption{Standard deviation of accuracy for all tasks of Figure~\ref{table:data-perf}. The left block are domain adaptation experiments with similar features, where both PGW and UGW are initialised with PW. The right block are domain adaptation experiments with different features, and the reported init is FLB~\eqref{eq-def-flb} used for UGW.}
	\label {table:data-std}
\end{table}
\clearpage
%\baselineskip
%\newpage
% \input{sections/checklist-neurips}

\end{document}